\documentclass[10pt]{article}

\usepackage{amsthm}
\usepackage{amsmath}
\usepackage{amsfonts}
\usepackage{dsfont}

\usepackage[all]{xy}
\usepackage{mathtools}
\usepackage{centernot}

\usepackage{ tipa }
\usepackage{ upgreek }
\usepackage{graphicx}
\usepackage{float}
\usepackage{amssymb}
\usepackage{epstopdf}
\usepackage{pgf,tikz}
\usepackage{pgfplots}
\usepackage{listings}
\usetikzlibrary{automata}
\usetikzlibrary{arrows}
\usetikzlibrary{positioning}
\usepackage{setspace}

\usepackage{pgf,tikz}
\usepackage{pgfplots}
\usetikzlibrary{automata}
\usetikzlibrary{arrows} 
\usetikzlibrary{positioning}
\usetikzlibrary{shapes,snakes}

\usepackage{calc}
\usepackage{tikz}
\usetikzlibrary{decorations.markings}
\tikzstyle{vertex}=[circle, draw, inner sep=0pt, minimum size=6pt]

\tikzstyle{stan}=[diamond, draw, inner sep=0pt, minimum size=6pt]

\theoremstyle{thmstyleone}
\newtheorem{ejem}{Example}
\newtheorem{defi}{Definition}
\newtheorem{teo}{Theorem}
\newtheorem{prop}{Proposition}
\newtheorem{lema}{Lemma}
\newtheorem{coro}{Corollary}

\numberwithin{ejem}{section}
\numberwithin{defi}{section}
\numberwithin{teo}{section}
\numberwithin{prop}{section}
\numberwithin{lema}{section}
\numberwithin{coro}{section}
\numberwithin{conjetura}{section}
\newtheorem{obs}{Remark}[section]

\usepackage{subfigure}

\newtheorem{remark}[ejem]{Remark}

\usepackage{graphicx} 

 \title{The kernel-subdivision number of a digraph}

\author{Teresa I. Hoekstra-Mendoza$^1$ \\ Miguel Licona-Vel\'azquez$^2$ \\ Roc\'io Rojas-Monroy$^3$ 
\\ \\
{\small $^2$ Universidad Autónoma Metropolitana-Iztapalapa } \\
{\small Mexico City, Mexico}\\
{\small {\tt eliconav23@xanum.uam.mx}}
\\
{\small $^1$ Centro de Investigaci\'on en Matem\'aticas, }
{\small Guajanajuato, Mexico}\\
{\small\tt maria.idskjen@cimat.mx}
\\
{\small $^3$Universidad Aut\'onoma del Estado de M\'exico} \\
{\small Toluca, Mexico}\\
{\small {\tt mrrm@uaemex.mx}} \\
\\
}
\date{December 2023}

\begin{document}

\maketitle

\begin{abstract}
    It is well known that determining if a digraph has a kernel is an NP-complete problem. However, Topp proved that when subdividing every arc of a digraph we obtain a digraph with a kernel. In this paper we define the kernel subdivision number $\kappa(D)$ of a digraph $D$ as the minimum number of arcs, such that, when subdividing them, we obtain a digraph with a kernel. We give a general bound for $\kappa(D)$ in terms of the number of directed cycles of odd length and compute $\kappa(D)$ for a few families of digraphs.
    If the digraph is $H$-colored, we can analogously define the $H$-kernel subdivision number. In this paper we also improve a result for $H$-kernels given by Galeana et al. to subdividing every arc of a spanning subgraph with certain properties.  Finally we prove that when the directed cycles of a digraph overlap little enough, we can obtain a good bound for the $H$-kernel subdivision number.
\end{abstract}
    

\section*{Introduction}
For general concepts we refer the reader to \cite{CL}.
Let $D$ be a digraph. A subset $N$ of $V(D)$ is said to be a \textit{kernel} if it is both independent (there are no arcs between any pair of vertices in $N)$ and absorbent (for all $u\in V(D)\setminus N$ there exists $v$ in $N$ such that $(u,v)\in A(D)).$ The concept of kernel was introduced in \cite{VNM} by von Newmann and Morgenstern in the context of Game Theory as a solution for cooperative $n$-players games. Kernels have been studied by several authors, see for example, \cite{HV}, \cite{Matus}, \cite{BG} and \cite{RV}. Chvátal proved in \cite{CV} that recognizing digraphs that have a kernel is an NP-complete problem, so finding sufficient conditions for a digraph to have a kernel or finding large families of digraphs with kernels are two lines of investigation widely studied by many authors. For example, in \cite{J} Topp defined the subdivision  digraph $S(D)$ of a digraph $D$ ($D$ possibly infinite), as follows.

\begin{defi}\label{def1}
Let $D$ be a digraph. The subdivision of $D,$ denoted by $S(D),$ is the digraph such that:
\begin{enumerate}
\item $V(S(D))= V(D)\cup A(D).$
\item $A(S(D))= \{(u,a): a=(u,v)\in A(D)\} \cup \{(a,v): a=(u,v)\in A(D)\},$ for each $a\in A(D).$ 
\end{enumerate}
\end{defi} 

Topp proved that $S(D)$ has at least one kernel for every digraph $D$, thus a natural question that one can ask is what is the smallest number of edges that need to be subdivided in so that the resulting digraph has a kernel.
We define the \textbf{kernel subdivision number} of a digraph $D$ as follows:
\begin{defi}
Given a digraph $D$, and a set $\Lambda \subset A(D)$, let $D_{\Lambda}$ denote the digraph obtained from $D$ by subdividing every arc of $\Lambda$.
We define the kernel-subdivision number of $D$, $\kappa(D)$ as the smallest cardinality of a set $\Lambda$ such that $D_{\Lambda}$ has a kernel.
\end{defi}

\begin{remark}\label{topp}
For every digraph $D$, $0 \leq \kappa(D) \leq |A(D)|$.
\end{remark}

\begin{prop}\label{ciclo}
For the directed cycle on length $n$ we have $$\kappa(C_n)= \left\{ \begin{array}{ccc}
   1  & \mbox{if} & n \mbox{ is odd}  \\
   0  &  \mbox{if} & n \mbox{ is even}
\end{array}\right..$$
\end{prop}
In Section \ref{S1}, we give an upper bound for $\kappa(D)$ which improves the bound obtained from Topp's result noticed in Remark \ref{topp} and give a an example of a digraph for which this bound coincides with its kernel subdivision number. We also do explicit computations for $\kappa(D)$ for several families of digraphs.
In the rest of the paper we consider a different variant of kernels, for which we require the following definitions.
\\

A digraph $D$ is said to be \textit{$m$-colored} if the arcs of $D$ are colored with $m$ colors. Let $D$ be an $m$-colored digraph, a path in $D$ is called \textit{monochromatic} if all of its arcs are colored alike. For an arc $(u,v)$ of $D$ we will denote its color by $c_{D}(u,v).$ Let $C=\{b_{1},\ldots,b_{m}\}$ be the set of colors used to color $A(D)$ and if $u$ is a vertex of $D,$ the set $\{b\in C:$ $c_{D}(u,v)=b$ for some $v$ in $V(D)\}$  will be denoted by $\xi _{D}^{+}(u).$ 

Let $D$ be an $m$-colored digraph, a subset $K$ of $V(D)$ is said to be a \textit{kernel by monochromatic paths} ($Kmp$) if it satisfies the following conditions: (1) there are no monochromatic paths between every pair of vertices in $K$ ($K$ is independent by monochromatic paths) and (2) for every vertex $x\in V(D)\setminus K$ there is an $xK$-monochromatic path ($K$ is absorbent by monochromatic paths). The notion of $Kmp$ was studied first in \cite{SSW} by Sands, Sauer and Woodrow.

The concept of kernel by monochromatic paths generalizes the concept of a kernel since an  kernel by monochromatic paths is a kernel when all the monochromatic paths have length one.

In \cite{HR} Galeana-Sánchez and Rojas-Monroy defined the subdivision digraph $S(D)$ of an $m$-colored digraph $D$, as follows.

\begin{defi}\label{defi2}
Let $D$ be an $m$-colored digraph. The subdivision of $D,$ denoted by $S(D),$ is the $m-colored$ digraph such that:
\begin{enumerate}
\item $V(S(D))= V(D) \cup A(D).$
\item $A(S(D))= \{(u,a): a=(u,v)\in A(D)\} \cup \{(a,v): a=(u,v)\in A(D)\},$ for each $a\in A(D).$ 
\item $S(D)$ is the $m-colored$ digraph where $(u,a,v)$ is a monochromatic path with the same color of $a=(u,v)\in A(D)$  with $\{u,v\}\subseteq V(D)$ and $a\in A(D).$
\end{enumerate}
\end{defi}

Galeana-Sánchez and Rojas-Monroy  proved that if $D$ has no monochromatic infinite outward path, then $S(D)$ has a $Kmp$.
\\
Let $H$ be a digraph possibly with loops, and $D$ be a digraph without loops whose arcs are colored with the vertices of $H,$ so $D$ is said to be an \textit{$H$-colored} digraph. A directed walk in $D$ is said to be an \textit{$H$-walk} if and only if the consecutive colors encountered on $W$ form a directed walk in $H.$ An $H$-path is an $H$-walk such that all the vertices are distinct, notice that an arc in $D$ is an $H$-path, that is to say, a singleton vertex is a walk in $H,$ and the concatenation of two $H$-walks is not always an $H$-walk. 
\\

In \cite{AL} Arpin and Linek defined the concept of $H$-kernel by walks as follows.
\\
Let $H$ be a digraph possibly with loops, and $D$ an $H$-colored digraph, subset $N$ of vertices of $D$ is said to be an $H$-kernel by walks if it holds both: (1) for every pair of vertices in $N$ there is no $H$-walk between them ($N$ is $H$-independent by walks) and (2) for each vertex $u$ in $V(D)\setminus N$ there exists an $H$-walk from $u$ to $N$ in $D$ ($N$ is $H$-absorbent by walks).  
\\

The concept of $H$-kernel generalizes the concept of a kernel and kernel by monochromatic paths since an $H$-kernel is a kernel when $A(H)=\emptyset$ and it is a kernel by monochromatic paths when $A(H)=\{(u,u):u\in V(H)\}.$        
\\

In \cite{HRR} Galeana-Sánchez, Rojas-Monroy, Sánchez-López and Zavala-Santana defined the subdivision digraph $S_{H}(D)$ with respect to $H$ of $D,$ as follows.

\begin{defi}\label{defi3}
Let $H$ and $D$ be two digraphs. A subdivision of $D,$ denoted by $S_{H}(D),$ is the $H$-colored digraph such that:
\begin{enumerate}
\item $V(S_{H}(D))= V(D) \cup A(D).$
\item $A(S_{H}(D))= \{(u,a): a=(u,v)\in A(D)\} \cup \{(a,v): a=(u,v)\in A(D)\},$ for each $a\in A(D).$ 
\item $S_{H}(D)$ is the $H$-colored digraph where $(u,a,v)$ is an $H$-walk with $\{u,v\}\subseteq V(D)$ and $a\in A(D).$
\end{enumerate}
\end{defi}

They proved that: Let $H$ be a digraph possibly with loops, $D$ a digraph without infinite outward paths and $S_{H}(D)$ a subdivision of $D$ with respect to $H.$ Suppose that $|V(D)|\geq 2k+3$ and $|\xi_{D}^{+}(v)|\leq k$ for every $v$ in $V(D)$, for some positive integer $k.$ Then, $S_{H}(D)$ has an $H$-kernel by walks. 

Section \ref{S2} is dedicated to proving that Galena-Sanchez, Rojas-Monroy, \linebreak Sánchez-López and Zavala-Santana's result holds for digraphs obtained from a digraph $D$ by subdividing every arc in a spanning subdigraph of $D.$  This leads to defining an analogue of the kernel subdivision number for $H$-kernels as follows. 
\begin{defi}
Let $H$ be a digraph possibly  with loops and $D$ an $H$-colored digraph. We define the  $H$-kernel subdivision number, denoted by $\kappa^H(D),$ to be the smallest number of arcs that need to be subdivided in $D$ in order to obtain a digraph with an $H$-kernel by walks.
\end{defi}
Finally in Section \ref{S3} we will improve the bound for $\kappa^H(D)$ obtained in Section \ref{S2}. This will be  for a specific family of graphs which turns out to be bounded by one and for when $A(H)=\{(u,u):u\in V(H)\},$ i.e we will obtain the kernel subdivision number by monochromatic paths.   

\section{A bound for $\kappa(G)$}\label{S1}
Richardson \cite{R} proved that every directed graph without odd directed cycles has a kernel. Naturally, we shall study  $\kappa(G)$ in terms of the number of cycles of $G,$ in order to get a bound for $\kappa(G)$  which improves the bound obtained in Remark \ref{topp}.
The following lemma is for undirected graphs, but it is a good warm up exercise which gives us the idea for the directed case.
\begin{lema}
Let $G$ be a planar connected graph without cut vertices, and assume $G$ has $C$ faces. Then we can subdivide at most $C$ edges to obtain a graph without odd cycles.
\end{lema}
\begin{proof}
    Since $G$ has no cut vertices, either $G=C_n$ for some $n\in \mathbb{N}$, or $G$ has an even amount of faces having an odd cycle as their boundary.
    Let $G^*$ be the dual graph of $G$, and $M$ be a matching of the vertices of $G^*$ which correspond to faces having odd cycles as their boundary in $G$. 
    For every pair $(a,b)\in M$, consider the shortest $ab$-path $\gamma$ and let $\Gamma_M$ be the set of all these paths.
    Now subdivide in $G$ every edge which belongs in $G^*$ to a path in $\Gamma$.
    Then every odd cycle in $G$ has now one edge subdivided and every even cycle has now an even amount of its edges subdivided. Notice that the total amount of subdivided edges is less or equal than the amount of cycles without chords.
\end{proof}


\begin{lema}\label{ciclos}
Let $D$ be a digraph without cut vertices, and let $\mathfrak{c}$ be the amount of directed cycles of $D$. Then $\kappa(D)\leq \mathfrak{c}.$
\end{lema}

\begin{proof}
    Let $D^*$ denote the graph which has as vertex set the set of all cycles without chords in $D$; both directed and undirected cycles. Two vertices are adjacent in $D^*$ if the corresponding cycles in $D$ share an edge.   We are going to give a colouring  $f: V(D^*)\rightarrow \{0,1,2\}$ of the vertices of $D^*$  as follows: 
    \begin{enumerate}
        \item The vertices corresponding to directed cycles of odd length in $D$  have color $1$,
        \item the vertices corresponding to directed cycles of even length in $D$  have color $2$, and
        \item the vertices corresponding to undirected cycles in $D$  have color $0$.
    \end{enumerate}
    Given a vertex $v$ of color $1$, consider a path $\gamma_v$ with the following properties:
    \begin{itemize}
        \item $\gamma_v$ starts at the vertex $v$,
        \item $\gamma_v$ ends at a vertex having color $1$ or $0$, and
        \item every vertex of $\gamma_v$ which is not an endpoint must have colour two.
    \end{itemize}
    Finally consider the set $\Gamma=\{\gamma_v: f(v)=1\}$ which contains exactly one path for each vertex $v$ of color 1 (here we are allowing $\gamma_v=\gamma_w$ if $w$ and $v$ are the endpoints of the path, both having color $1$), thus $|\Gamma| \leq \{v \in V(D^*): v \mbox{ has color }1 \}$.
    Now subdivide every edge of $D$ which corresponds in $D^*$ to an edge of a path in $\Gamma.$ Since every path has at most one vertex of color $0$, the number of subdivided edges is at most $\mathfrak{c}$. Finally, notice that we subdivided one edge of each odd directed cycle and $2k$ edges of every even cycle where $k$ is the number of paths in $\Gamma$ containing the vertex corresponding to the even cycle. 
\end{proof}
\begin{ejem}
    For example consider the digraph of Figure \ref{cc} (left). We can see in Figure \ref{cc} (right) the digraph $D^{\ast}$ with its vertices colored where black is color two, gray is color one, and white is color zero. The dashed edges correspond to the set $\Gamma$.
\end{ejem}

\begin{figure}[h!]
	\centering
	\begin{tikzpicture}[every node/.style={circle, draw, scale=.6}, scale=1.0, rotate = 180, xscale = -1]

		\node (1) at ( 2.48, 2.3) {};
		\node (2) at ( 4.15, 2.3) {};
		\node (3) at ( 4.15, 3.8) {};
		\node (4) at ( 2.48, 3.8) {};
		\node (5) at ( 3.3, 1.0) {};
		\node (6) at ( 1.15, 3.21) {};
		\node (7) at ( 5.49, 3.19) {};
		\node (8) at ( 3.55, 5.18) {};
		\node[fill=black] (9) at ( 9.0, 3.2) {};
		\node[fill=gray] (10) at ( 9.0, 2.2) {};
		\node[fill=gray] (11) at ( 8.0, 3.2) {};
		\node[fill=gray] (12) at ( 10, 3.2) {};
		\node[fill=gray] (13) at ( 9, 4.2) {};
		\node (14) at ( 10, 2.2) {};
		\node (15) at ( 8, 2.2) {};
		\node (16) at ( 8, 4.2) {};
		\node (17) at ( 10, 4.2) {};

		\draw[->] (1) -- (2);
		\draw[->] (4) -- (1);
		\draw[->] (6) -- (4);
		\draw[->] (3) -- (4);
		\draw[->] (8) -- (3);
		\draw[->] (4) -- (8);
		\draw[->] (2) -- (3);
		\draw[->] (3) -- (7);
		\draw[->] (8) -- (7);
		\draw[->] (6) -- (1);
		\draw[->] (1) -- (5);
		\draw[->] (5) -- (2);
		\draw[->] (2) -- (7);
		\draw[->] (7) -- (5);
		\draw[->] (6) -- (5);
		\draw[->] (8) -- (6);
		\draw (10) -- (9);
		\draw (11) -- (9);
		\draw (12) -- (9);
		\draw (13) -- (9);
		\draw (14) -- (12);
		\draw[dashed] (10) -- (14);
		\draw (15) -- (10);
		\draw[dashed] (11) -- (15);
		\draw (16) -- (11);
		\draw[dashed] (13) -- (16);
		\draw (17) -- (13);
		\draw[dashed] (12) -- (17);

	\end{tikzpicture}
	\caption{The digraphs $D$ and $D^{\ast}$}
	\label{cc}
\end{figure}
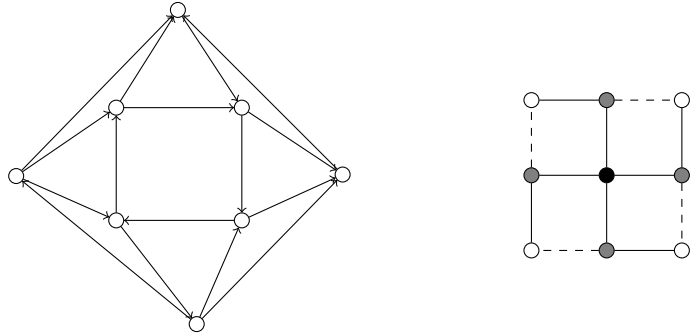

 \begin{teo}
For a digraph $D,$ let $cic(D)$ denote the number of directed cycles without chords.
Then the following inequality holds:
$$0 \leq \kappa(D) \leq cic(D).$$
 \end{teo}
\begin{proof}
    Let $D_1, \dots, D_k$ be the maximal induced subgraphs of $D$ which do not have cut vertices.
    Apply Lemma \ref{ciclos} to every subgraph $D_i.$
\end{proof}

This improves the inequality previously known from Topp's result, but we are going to give both a digraph for which this bound is exact, and a digraph having an arbitrarily large number of directed cycles, and having kernel-subdivision number one. 

Let $cic(D)$ denoted the number of directed cycles in $D$.

\begin{prop}\label{1}
Denote by $R_n$ the digraph having $V(D)= \{x,y,z_1, \dots, z_n\}$ and 
$A(D) = (y,x)\bigcup \limits_{i=1}^n \{(x,z_i), (z_i,y)\}$. Then $\kappa(R_n)=1$ for $n\in\mathbb{N}$,
\end{prop}
\begin{proof}
Since $C_{3}$ does not have a kernel, we have that $R_n$ does not have a kernel but by subdividing the arc $(y,x)$, the set $\{x,y\}$ is a kernel of $R_{n_{\Lambda}}$ where $\Lambda=\{(y,x)\}$.

\begin{figure}[h!]
\centering

\begin{tikzpicture}
\node[style={circle, draw, scale=.5}] (x) at (0,0){};
\node[style={circle, draw, scale=.5}] (y) at (2,0){};
\node[style={circle, draw, scale=.5}] (1) at (1,1){};
\node[style={circle, draw, scale=.5}] (2) at (1,2){};
\node[style={circle, draw, scale=.5}] (n) at (1,4){};
\draw[dotted] (2)--(n);
\draw[->] (y)--(x);
\draw[->] (x)--(1);
\draw[->] (x)--(2);
\draw[->] (x)--(n);
\draw[->] (1)--(y);
\draw[->] (2)--(y);
\draw[->] (n)--(y);

\node[style={circle, draw, scale=.5}] (0x) at (4,0){};
\node[style={circle, draw, scale=.5}] (0y) at (6,0){};
\node[style={circle, draw, scale=.5}] (01) at (5,1){};
\node[style={circle, draw, scale=.5}] (02) at (5,2){};
\node[style={circle, draw, scale=.5}] (0n) at (5,4){};
\node[style={circle, draw, scale=.5}] (0z) at (5,0){};
\draw[dotted] (02)--(0n);
\draw[->] (0y)--(0z);
\draw[->] (0z)--(0x);
\draw[->] (0x)--(01);
\draw[->] (0x)--(02);
\draw[->] (0x)--(0n);
\draw[->] (01)--(0y);
\draw[->] (02)--(0y);
\draw[->] (0n)--(0y);
\end{tikzpicture}
\caption{The digraph $R_n$ and $R_{n_{\Lambda}}$ obtained by the subdivision of $\Lambda=\{(y,x)\}$}
    \label{Rn}
 \end{figure}
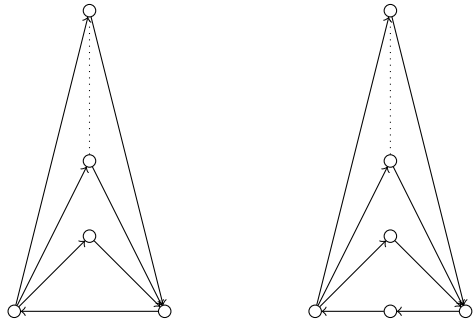

\end{proof}

\begin{prop}\label{s}
Consider the directed star graph $K_{1,n}$ where every edge is oriented towards the leaves. Let $S_n$ be a graph obtained by substituting every leaf of $K_{1,n}$ with an odd directed cycle. 
Then, for every integer $n \geq 2,$ $\kappa(S_n) = cic(S_n)$
\end{prop}
\begin{proof}
    Since every odd cycle $\gamma$ is a terminal strongly connected component, there are no possible vertices which can absorb a vertex of $\gamma.$ This means that the vertices of $\gamma$ must absorb each other. Since no two cycles share an edge, we can consider these strongly connected components separately and by Proposition \ref{ciclo}, we have that $\kappa(G)=n.$
\end{proof}

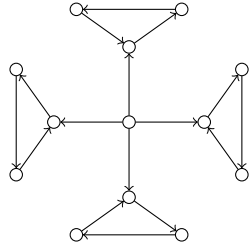
\begin{figure}[h!]
    \centering
  \begin{tikzpicture}  
     \node[style={circle, draw, scale=.5}](0) at (0,0){};
      \node[style={circle, draw, scale=.5}](y) at (1,0){};
      \node[style={circle, draw, scale=.5}](x) at (0,1){};
      \node[style={circle, draw, scale=.5}](z) at (-1,0){};
      \node[style={circle, draw, scale=.5}] (w) at (0,-1){};
      
      \node[style={circle, draw, scale=.5}] (y1) at (1.5,0.7){};
      \node[style={circle, draw, scale=.5}](y2) at (1.5,-0.7){};
      \node[style={circle, draw, scale=.5}] (x1) at (0.7, 1.5){};
      \node[style={circle, draw, scale=.5}] (x2) at (-0.7,1.5){};
      \node[style={circle, draw, scale=.5}] (z1) at (-1.5,0.7){};
      \node[style={circle, draw, scale=.5}](z2) at (-1.5,-0.7){};
      \node[style={circle, draw, scale=.5}] (w1) at (0.7, -1.5){};
      \node[style={circle, draw, scale=.5}] (w2) at (-0.7,-1.5){};
      
\draw[->] (0)--(y);
\draw[->] (0)--(w);
\draw[->] (0)--(x);
\draw[->] (0)--(z);
\draw[->] (x)--(x1);
\draw[->] (x1)--(x2);
\draw[->] (x2)--(x);
\draw[->] (y)--(y1);
\draw[->] (y1)--(y2);
\draw[->] (y2)--(y);
\draw[->] (z)--(z1);
\draw[->] (z1)--(z2);
\draw[->] (z2)--(z);
\draw[->] (w)--(w1);
\draw[->] (w1)--(w2);
\draw[->] (w2)--(w);
  \end{tikzpicture}
    \caption{A digraph $S_4$.}
    \label{Sn}
\end{figure}

Combining the digraphs obtained in propositions $\ref{1}$ and \ref{s} we obtain the following result. 
\begin{teo}
Let $n,m \in \mathbb{N}$ with $n\leq m$ there exists a graph $G$ such that $cic(G)=m$ and $\kappa(G)=n.$
\end{teo}
\begin{proof}
Let $m_1, \dots, m_n$ be integers such that $\sum\limits_{i=1}^nm_i=m.$ Consider the directed star graph $K_{1,n}$ where each edge is oriented towards the leaves and let $\{v_1, \dots, v_n\}$ be the leaves of $K_{1,n}$. Let $G$ be the digraph obtained from $K_{1,n}$ by substituting the vertex $v_i$ with the graph $R_{m_i}$ for $1\leq i \leq n.$ Then $G$ is the desired graph.
\end{proof}
\begin{ejem}
In Figure \ref{27} we can see a digraph $G$ having $7$ directed cycles of odd length and $\kappa(G)=2.$
\end{ejem}
\begin{figure}[h!]
	\centering
	\begin{tikzpicture}[every node/.style={circle, draw, scale=.5}, scale=1.0, rotate = 180, xscale = -1]

		\node (1) at ( 4.22, 3) {};
		\node (2) at ( 3.4, 2.06) {};
		\node (3) at ( 2.43, 3.0) {};
		\node (4) at ( 6.12, 3) {};
		\node (5) at ( 7.76, 3) {};
		\node (6) at ( 9.61, 3) {};
		\node (7) at ( 8.6, 3.72) {};
		\node (10) at ( 8.6, 1.9) {};
		\node (11) at ( 3.4, 1.01) {};
		\node (12) at ( 3.4, 3.93) {};
		\node (13) at ( 8.6, 1.08) {};
		\node (14) at ( 8.6, 4.68) {};

		\draw[->] (4) -- (5);
		\draw[->] (4) -- (1);
		\draw[->] (6) -- (5);
		\draw[->] (7) -- (6);
		\draw[->] (5) -- (7);
		\draw[->] (3) -- (2);
		\draw[->] (1) -- (3);
		\draw[->] (10) -- (6);
		\draw[->] (2) -- (1);
		\draw[->] (11) -- (1);
		\draw[->] (3) -- (11);
		\draw[->] (3) -- (12);
		\draw[->] (12) -- (1);
		\draw[->] (13) -- (6);
		\draw[->] (5) -- (10);
		\draw[->] (5) -- (13);
		\draw[->] (14) -- (6);
		\draw[->] (5) -- (14);

	\end{tikzpicture}
	\caption{A digraph $G$ containing 7 directed cycles of odd length having $\kappa(G)=2.$ }
	\label{27}
\end{figure}
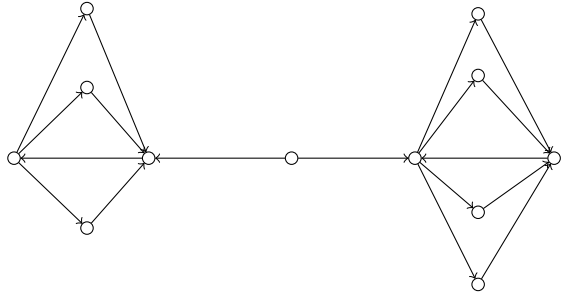

\begin{teo}
Let $G$ be a digraph having $k$ terminal strongly connected components $g_1, \dots, g_k.$
Then $\kappa(G)\geq \sum\limits_{i=1}^k \kappa(g_i).$
\end{teo}

\subsection{The kernel subdivision number of certain families of graphs}
In this section we shall analyze various families of digraphs, all of which have large amounts of directed cycles. We shall start studying the kernel subdivision number for a family of digraphs called tournaments (an orientation of the complete graph). Given a digraph $D$, \textit{$\alpha(D)$} denotes the minimal cardinality of an absorbent set for $V(D)$.

\begin{teo}
Let $T$ be a tournament. Then, $$\kappa(T)= \sum \limits_{i=1}^{\alpha(T)-1}i.$$
\end{teo}
\begin{proof}
Let $K\subseteq V(T)$ be an  absorbent set of size $\alpha(T)$. If $\alpha(T)=1$ we have that that this set is also an independent set. Then $\kappa(T)=0.$ Thus we can assume that  $\alpha(T)\geq 2.$ 

Notice that, $T'$ the induced subdigraph spanned by $K$ is again a tournament, thus it has $\sum\limits_{i=1}^{\alpha(T)-1}i$ arcs, let $\Lambda=A(T')$. If we subdivide every arc of $T'$, we obtain that $K$ is a kernel in $T_{\Lambda}$, thus $\kappa(T) \leq \sum \limits_{i=1}^{\alpha(T)-1}i .$

Now, suppose that $\kappa (T) < \sum \limits_{i=1}^{\alpha(T)-1}i.$ Then there exists $\Lambda'\subseteq A(T)$ such that $\kappa (T)=|\Lambda'|.$ Let $I \subseteq V(T)$ such that $x,y\in I$ if and only if $(x,y)$ or $(y,x) $ is in $\Lambda'$. Let $T_{0}$ be the induced subdigraph spanned by $I,$  since $T$ is a tournament we have that $A(T_{0})=\Lambda'.$  Recall that $K$ has minimal cardinality, so  if $I$ is an absorbent set then $|I|\geq|K|.$ Therefore $|\Lambda'|\geq \sum \limits_{i=1}^{\alpha(T)-1}i$ which is a contradiction. 

So, let us suppose that $I$ is not an absorbent set, then there exists  $I'\subseteq V(T)$ such that $I\cup I'$ is an absorbent set in $T$. Let $T_{1}$ be the induced subdigraph spanned by $I\cup I'$ and $\Lambda'=A(T_{1}).$ Notice that $I\cup I'$ is a kernel in $T_{\Lambda'}.$ Since $T$ is a tournament, there exist $x\in I$ and $y\in I'$ such that $(x,y)$ or $(y,x)$ is an element of $A(T)$ but it is not an element of $\Lambda'.$ This means that $I\cup I'$ is not an independent set in $T_{\Lambda'}$ which is a contradiction. Therefore $\kappa (T)> |\Lambda'|,$ since $K$ has minimal cardinality, we have that $|I\cup I'|\geq |K|$  and hence $\kappa (T)\geq \sum \limits_{i=1}^{\alpha(T)-1}i$ which is a contradiction. 

Thus, $\kappa(T)= \sum \limits_{i=1}^{\alpha(T)-1}i.$
 
\end{proof}

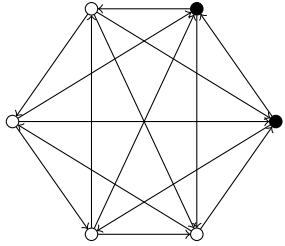
\begin{figure}[h!]
    \centering
   \begin{tikzpicture}
       \node[circle, draw, scale=.5] (0) at (-.25,0){};
       \node[circle, draw, scale=.5] (1) at (.8,1.5){};
       \node[circle, draw, scale=.5, fill=black] (2) at (2.2,1.5){};
       \node[circle, draw, scale=.5, fill=black] (4) at (3.25,0){};
       \node[circle, draw, scale=.5] (5) at (2.2,-1.5){};
       \node[circle, draw, scale=.5] (6) at (.8,-1.5){};
       \draw[->] (1)to(0);
       \draw[->] (0)to(6);
       \draw[->] (6)to(5);
       \draw[->] (5)to(4);
       \draw[->] (4)to(2);
       \draw[->] (2)to(1);
       \draw[->] (0)to(2);
       \draw[->] (1)to(4);
       \draw[->] (2)to(5);
       \draw[->] (4)to(6);
       \draw[->] (5)to(0);
       \draw[->] (0)to(4);
       \draw[->] (6)to(2);
       \draw[->] (1)to(5);
       \draw[->] (6)to(1);
   \end{tikzpicture}
    \caption{A tournament with six vertices having kernel subdivision number one}
\end{figure}

Consider the triangular directed grid  $T(m,n)$ graph defined as follows
$$V(T(m,n)))=\bigcup\limits_{i=1}^m\bigcup\limits_{j=1}^n x_i^j$$
\begin{align*}
    A(T(m,n)) =& \bigcup \limits_{j=1}^n \bigcup \limits_{i=1}^{m-1} (x_i^j,x_{i+1}^j)\cup \\
    &\bigcup \limits_{j=1}^{n-1} (\bigcup \limits_{i=1}^{m-1} \{(x_i^{j+1},x_i^j), (x_i^j,x_{i-1}^{j+1}\}\cup (x_m^j, x_m^{j+1}))
\end{align*}
For example, in figures \ref{73} and \ref{72} we can see the digraphs $T(7,3)$ and $T(7,2)$ respectively.

\begin{figure}[h!]
    \centering
\begin{tikzpicture}
\node[circle, draw, scale=.4] (11) at (0,0){};
\node[circle, draw, scale=.4, fill=black] (12) at (1,0){};
\node[circle, draw, scale=.4] (13) at (2,0){};
\node[circle, draw, scale=.4, fill=black] (14) at (3,0){};
\node[circle, draw, scale=.4] (15) at (4,0){};
\node[circle, draw, scale=.4, fill=black] (16) at (5,0){};
\node[circle, draw, scale=.4, fill=black] (17) at (6,0){};
\draw[->] (11)--(12);
\draw[->] (12)--(13);
\draw[->] (13)--(14);
\draw[->] (14)--(15);
\draw[->] (15)--(16);
\draw[->] (16)--(17);

\node[circle, draw, scale=.4] (21) at (1,1){};
\node[circle, draw, scale=.4] (22) at (2,1){};
\node[circle, draw, scale=.4] (23) at (3,1){};
\node[circle, draw, scale=.4] (24) at (4,1){};
\node[circle, draw, scale=.4] (25) at (5,1){};
\node[circle, draw, scale=.4] (26) at (6,1){};
\node[circle, draw, scale=.4] (27) at (7,1){};
\draw[->] (21)--(22);
\draw[->] (22)--(23);
\draw[->] (23)--(24);
\draw[->] (24)--(25);
\draw[->] (25)--(26);
\draw[->] (26)--(27);

\node[circle, draw, scale=.4, fill=black] (31) at (0,2){};
\node[circle, draw, scale=.4] (32) at (1,2){};
\node[circle, draw, scale=.4, fill=black] (33) at (2,2){};
\node[circle, draw, scale=.4] (34) at (3,2){};
\node[circle, draw, scale=.4,fill=black] (35) at (4,2){};
\node[circle, draw, scale=.4] (36) at (5,2){};
\node[circle, draw, scale=.4,fill=black] (37) at (6,2){};
\draw[->] (31)--(32);
\draw[->] (32)--(33);
\draw[->] (33)--(34);
\draw[->] (34)--(35);
\draw[->] (35)--(36);
\draw[->] (36)--(37);

\draw[->] (21)--(11);
\draw[->] (22)--(12);
\draw[->] (23)--(13);
\draw[->] (24)--(14);
\draw[->] (25)--(15);
\draw[->] (26)--(16);
\draw[->] (27)--(17);

\draw[<-] (31)--(21);
\draw[<-] (32)--(22);
\draw[<-] (33)--(23);
\draw[<-] (34)--(24);
\draw[<-] (35)--(25);
\draw[<-] (36)--(26);
\draw[<-] (37)--(27);

\draw[->] (12)--(21);
\draw[->] (13)--(22);
\draw[->] (14)--(23);
\draw[->] (15)--(24);
\draw[->] (16)--(25);
\draw[->] (17)--(26);

\draw[->] (32)--(21);
\draw[->] (33)--(22);
\draw[->] (34)--(23);
\draw[->] (35)--(24);
\draw[->] (36)--(25);
\draw[->] (37)--(26);
\end{tikzpicture}
    \caption{The triangular grid graph $T(7,3)$ with $\kappa(T(7,3))=1$}
    \label{73}
\end{figure}
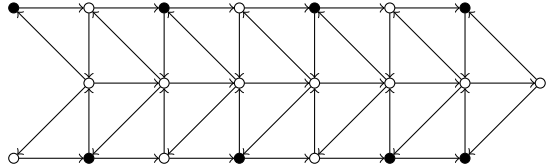
\begin{prop}\label{T(n,3)}
    For every $n>1,m=3,5$ $\kappa(T(n,m))=1$  
\end{prop}
\begin{proof}
    Assume that $T(n,3)$ has a kernel $N.$ Notice that at least one of the vertices $x_1^n, x_2^n,x_3^n$ belongs to $N.$ If both $x_1^n$ and $x_3^n$ belong to $N$ then there cannot be a vertex in $N$ which absorbs $x_2^{n-1}.$ If $x_2^n\in N$ then there cannot be a vertex in $N$ which absorbs $x_1^n$ and $x_3^n.$ If $x_1^n\in N$ there is no vertex in $N$ which absorbs $x_3^n$ and vice versa. This is a contradiction thus $\kappa(T(n,3))\geq 1.$
    Now consider the set $$N=\{ x_1^n, x_1^{n-2}, \dots, x_1^{n-2 \lfloor \frac{n}{2} \rfloor}, x_3^n, x_3^{n-1}, x_3^{n-3}, \dots, x_3^{n-1-2 \lfloor \frac{n}{2} \rfloor}\}.$$
Notice that $N$ is a kernel for $T_{\Lambda}(n,3)$ obtained from $T(n,3)$ by subdividing the edge in $\Lambda=\{(x_3^n, x_3^{n-1})\}$ thus $\kappa (T(n,3))=1.$
\\

By construction we have that $T(n,3)$ is an induced subdigraph of $T(n,5).$ Thus, we can follow a similar procedure as in the previous case in order to obtain that $\kappa(T(n,5))\geq 1.$ 
Now consider the set $N=\{ x_1^n, x_1^{n-2}, \dots,$ $ x_1^{n-2 \lfloor \frac{n}{2} \rfloor},x_3^n, x_3^{n-1},$ $x_3^{n-3}, \dots, x_3^{n-1-2 \lfloor \frac{n}{2} \rfloor},  x_5^n, x_5^{n-2}, \dots, x_5^{n-2 \lfloor \frac{n}{2} \rfloor}\},$  it is easy to verify that $N$ is a kernel of $T_{\Lambda}(n,5)$ obtained by subdividing the edge in $\Lambda=\{(x_3^n, x_3^{n-1})\}$ thus $\kappa (T(n,5))=1.$
\end{proof}

\begin{prop}\label{t24}
    For every $n>1,$ $\kappa(T(n,m))=\lceil \frac{n}{2} \rceil $ for $m=2,4$.
\end{prop}
\begin{proof}
    Take $I$ to be a maximal independent set contained in the second row if $m=2$, or a maximal independent set contained in the second and fourth row if $m=4$. Then there exists $J\subset V(T(n,m))$ such that:
    \begin{itemize}
        \item $J$ is an independent set with $|J|=\lceil \frac{n}{2}\rceil\frac{m}{2}  $,
        \item $I$ absorbs $V(G)\setminus J,$ and
        \item  there exists an edge $e$ such that $I$ is a kernel for $T(n,m)_e$
        
    \end{itemize}
   This means that $\kappa(T(n,m)) \leq \lceil \frac{n}{2} \rceil. $
    Any other maximal independent set $I'$ contains vertices from two consecutive rows and in particular $I'$ absorbs less vertices as $I.$ The vertices which are not absorbed by $I'$, form again an independent set. Notice that since $l-1$ of the $l$ vertices which are not absorbed by $I'$ can be put into pairs such that each pair has a common outer neighbour, the number of edges which need to be subdivided in order to extend $I'$ to a kernel is  $\lceil \frac{l-1}{2} \rceil+1> 1.$ (see fore example Figure \ref{72} where the black vertices form the set $I'$ for $n=7$). Hence $\kappa(T(n,m))=\lceil \frac{n}{2} \rceil $. 
\end{proof}

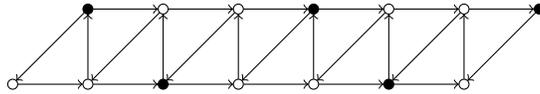
\begin{figure}[h!]
    \centering
\begin{tikzpicture}
\node[circle, draw, scale=.4] (11) at (0,0){};
\node[circle, draw, scale=.4, ] (12) at (1,0){};
\node[circle, draw, scale=.4, fill=black] (13) at (2,0){};
\node[circle, draw, scale=.4,] (14) at (3,0){};
\node[circle, draw, scale=.4] (15) at (4,0){};
\node[circle, draw, scale=.4, fill=black] (16) at (5,0){};
\node[circle, draw, scale=.4, ] (17) at (6,0){};
\draw[->] (11)--(12);
\draw[->] (12)--(13);
\draw[->] (13)--(14);
\draw[->] (14)--(15);
\draw[->] (15)--(16);
\draw[->] (16)--(17);

\node[circle, draw, scale=.4, fill=black] (21) at (1,1){};
\node[circle, draw, scale=.4] (22) at (2,1){};
\node[circle, draw, scale=.4] (23) at (3,1){};
\node[circle, draw, scale=.4,fill=black] (24) at (4,1){};
\node[circle, draw, scale=.4] (25) at (5,1){};
\node[circle, draw, scale=.4] (26) at (6,1){};
\node[circle, draw, scale=.4,fill=black] (27) at (7,1){};
\draw[->] (21)--(22);
\draw[->] (22)--(23);
\draw[->] (23)--(24);
\draw[->] (24)--(25);
\draw[->] (25)--(26);
\draw[->] (26)--(27);
\draw[->] (12)--(21);
\draw[->] (13)--(22);
\draw[->] (14)--(23);
\draw[->] (15)--(24);
\draw[->] (16)--(25);
\draw[->] (17)--(26);
\draw[->] (21)--(11);
\draw[->] (22)--(12);
\draw[->] (23)--(13);
\draw[->] (24)--(14);
\draw[->] (25)--(15);
\draw[->] (26)--(16);
\draw[->] (27)--(17);
\end{tikzpicture}
\caption{The digraph $T(7,2)$.}\label{72}
\end{figure}

\begin{teo}
For $m\geq 3,n>1$,  we have $$\kappa(T(n,m))= \left\{ \begin{array}{cc}
    k & \mbox{ if } m=4k\pm1 \\
    k + \lceil \frac{n}{2} \rceil & \mbox{ if } m=4k \text{ or }m=4k+2 \\
  
\end{array} \right.$$
\end{teo}

\begin{proof}
  Proceed by induction over $k$.   The base cases are due to propositions \ref{T(n,3)} and \ref{t24}. Assume $T(n,m)_{\Lambda}$ has a kernel $K$ for some set of edges $\Lambda.$ Let $l$ be the number of edges in $\Lambda$ which belong to the subgraph of $T(n,m)$ spanned by the last three rows. Let $d$ denote the number of edges in $\Lambda$ which have one endpoint in the $(m-3)$th row and the other endpoint in the $(m-4)$th row.
Let $L$ be the subset of $K$ consisting of vertices which belong to these three last rows together with the new vertices obtained from subdividing one of the $l$ edges. Notice that $N^-(L)$ is contained in the last four rows. Consider  $T(n,m)_{\Lambda}\setminus N^-[L]$ and notice that it is a subdivision $H_{\Lambda_0}$ of the graph $H$, which is obtained by adding inner leaves to $T(n,m-4).$
The set $K\setminus L$ is a kernel for $H_{\Lambda_0}$ and  let $M$ denote the subset of $K\setminus L$ which does not contain the inner leaves.
Since these inner leaves have outdegree zero, $M$ is a kernel 
 for $T(n,m-4)_{\Lambda_1}$ where $\Lambda_1$ is obtained from $\Lambda_0$ by removing the edges inciding in the inner leaves. If $k$ is even, by induction hypothesis, $|\Lambda_1|\geq k-1 + \lceil \frac{n}{2} \rceil$, thus $|\Lambda|=l+d+|\Lambda_0|\geq l + k-1 + \lceil \frac{n}{2} \rceil.$ If $k$ is odd, by induction hypothesis, $|\Lambda_1|\geq k-1$ thus $|\Lambda|=l+d+|\Lambda_0|\geq l + k-1$.
 Finally, $l+d>0$ since otherwise we can find a kernel in $T(3,n)$ which contradicts Proposition \ref{T(n,3)}.
\end{proof}
 
    Consider the following digraph $G$ which we shall call a $(p_1, \dots, p_{n-1})$-pithaya graph (or dragon fruit graph); $$V(G)=\{a_1,\dots ,a_n\}\cup_{i=1}^{n-1}\cup_{j=1}^{p_i}\{x_i^j,y_i^j,z_i^j\}\},$$
    $$A(G)=$$
    
    $\left\{ \begin{array}{cc}
      (a_i,x_j^i),(x_j^i,y_j^i),(a_{i+1},z_j^i),(z_j^i,y_j^i)   & \mbox{ for }j \mbox{ odd }, 1\leq j \leq p_i, 1\leq i \leq n-1 \\
    (x_j^i,a_i),(y_j^ix_j^i),(z_j^i,a_{i+1}),(y_j^i,z_j^i)   & \mbox{ for }j \mbox{ even}, 1\leq j \leq p_i,1\leq i \leq n-1
    \end{array}\right.$ 
    
   \begin{teo}
Let $G$ be a $(p_1, \dots, p_{n-1})$-pithaya graph. Then $\kappa(G)=n$ if $n_i >2.$
   \end{teo}
   \begin{proof}
       If we subdivide the arcs $(a_i,x_1^i)$ for $1 \leq i \leq n-1$ and $(a_n, z_1^{n-1})$, the obtained graph has a kernel thus $\kappa(G)\leq n.$
       For the other inequality we will proceed by induction over $n.$ Take $n=2$ and assume that $G$ has a kernel $N.$ Then $N$ contains at least one of the three vertices $x_1^1, y_1^1$ and $z_1^1.$ If $N$ contains $y_1^1$ then $a_1, a_2 \in N.$ But then there is no possible vertex in $N$ which absorbs $x_2^1$ and $z_2^1.$
    Now assume that $x_1^1\in N.$ Then $N$ contains an outer neighbor of $a_1,$ say $y_2^1.$ This means that $ y_3^1\in N$ but then there is no vertex which absorbs $x_3^1$ (here we can analyze further to see that either $z_3^1$ or $z_2^1$ is also not absorbed by $N$), a contradiction. Hence $\kappa(G) >0.$ Notice that, in every case there are at least two non adjacent vertices which cannot be absorbed by an independent set. Moreover these two vertices never have a common outer neighbour.  This means that we need to subdivide at least two edges to obtain a digraph with a kernel.

    Now assume that $G$ is a $(p_1, \dots, p_{n-1})$-pithaya graph. If $\kappa(G)=n-1$, let $\Lambda$ be the set of $n-1$ edges  such that when subdividing them we obtain a digraph with a kernel. Notice that there is a $p_i$-pithaya subgraph $P$ such that $E(P)\cap \Lambda = \emptyset.$ In particular, this means that both $a_i$ and $a_{i+1}$ are absorbed in the subdivided graph by vertices outside of $P.$ Then $D=(G\setminus V(P))\cup \{a_i, a_{i+1}\}$ is a disjoint union of two pithaya graphs having $\kappa(D)=n-1$ which contradicts the induction hypothesis.
   \end{proof}

\begin{figure}[h!]
\centering
\begin{tikzpicture}
\node[circle,draw,scale=.4] (a1) at (0,0){};
\node[circle,draw,scale=.4] (11x) at (1,0.5){};
\node[circle,draw,scale=.4] (12x) at (1,-0.5){};
\node[circle,draw,scale=.4] (14x) at (1,1.5){};
\node[circle,draw,scale=.4] (13x) at (1,-1.5){};
\node[circle,draw,scale=.4] (11y) at (2,0.5){};
\node[circle,draw,scale=.4] (12y) at (2,-0.5){};
\node[circle,draw,scale=.4] (14y) at (2,1.5){};
\node[circle,draw,scale=.4] (13y) at (2,-1.5){};
\node[circle,draw,scale=.4] (11z) at (3,0.5){};
\node[circle,draw,scale=.4] (12z) at (3,-0.5){};
\node[circle,draw,scale=.4] (14z) at (3,1.5){};
\node[circle,draw,scale=.4] (13z) at (3,-1.5){};
\node[circle,draw,scale=.4] (a2) at (4,0){};
\draw[->,bend left] (a1)to(11x);
\draw[->,bend right] (a1)to(13x);
\draw[->,bend left] (12x)to(a1);
\draw[->,bend right] (14x)to(a1);
\draw[->,bend left] (11x)to(11y);
\draw[->,bend right] (13x)to(13y);
\draw[->,bend left] (12y)to(12x);
\draw[->,bend right] (14y)to(14x);
\draw[->] (11y) to (12y);
\draw[->] (13y) to (14y);
\draw[->] (13y) to (12y);
\draw[->,bend right] (a2)to(11z);
\draw[->,bend left] (a2)to(13z);
\draw[->,bend right] (12z)to(a2);
\draw[->,bend left] (14z)to(a2);
\draw[->,bend right] (11z)to(11y);
\draw[->,bend left] (13z)to(13y);
\draw[->,bend right] (12y)to(12z);
\draw[->,bend left] (14y)to(14z);

\node[circle,draw,scale=.4] (21x) at (5,0.5){};
\node[circle,draw,scale=.4] (22x) at (5,-0.5){};
\node[circle,draw,scale=.4] (24x) at (5,1.5){};
\node[circle,draw,scale=.4] (23x) at (5,-1.5){};
\node[circle,draw,scale=.4] (21y) at (6,0.5){};
\node[circle,draw,scale=.4] (22y) at (6,-0.5){};
\node[circle,draw,scale=.4] (24y) at (6,1.5){};
\node[circle,draw,scale=.4] (23y) at (6,-1.5){};
\node[circle,draw,scale=.4] (21z) at (7,0.5){};
\node[circle,draw,scale=.4] (22z) at (7,-0.5){};
\node[circle,draw,scale=.4] (24z) at (7,1.5){};
\node[circle,draw,scale=.4] (23z) at (7,-1.5){};
\node[circle,draw,scale=.4] (a3) at (8,0){};
\draw[->,bend left] (a2)to(21x);
\draw[->,bend right] (a2)to(23x);
\draw[->,bend left] (22x)to(a2);
\draw[->,bend right] (24x)to(a2);
\draw[->,bend left] (21x)to(21y);
\draw[->,bend right] (23x)to(23y);
\draw[->,bend left] (22y)to(22x);
\draw[->,bend right] (24y)to(24x);
\draw[->] (21y) to (22y);
\draw[->] (23y) to (24y);
\draw[->] (23y) to (22y);
\draw[->,bend right] (a3)to(21z);
\draw[->,bend left] (a3)to(23z);
\draw[->,bend right] (22z)to(a3);
\draw[->,bend left] (24z)to(a3);
\draw[->,bend right] (21z)to(21y);
\draw[->,bend left] (23z)to(23y);
\draw[->,bend right] (22y)to(22z);
\draw[->,bend left] (24y)to(24z);
\end{tikzpicture}
\caption{A $(4,4)$ pithaya graph }
\end{figure}
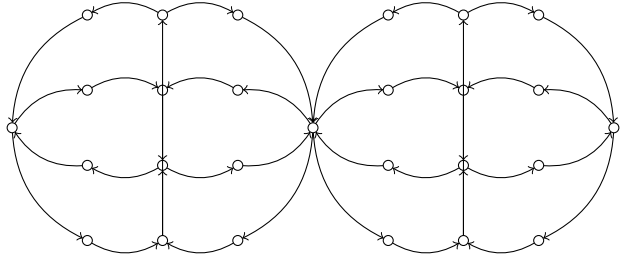

\section{A bound for $\kappa^{H}(G)$}\label{S2}
The following result was obtained by Galeana-Sánchez, Rojas-Monroy, \linebreak Sánchez-López and Zavala-Santana in \cite{HRR}.
 \begin{teo}\label{Hsubdivision}\cite{HRR}
      Let $H$ be a digraph possibly with loops, $D$ a digraph without infinite outward paths and $S_{H}(D)$ a subdivision of $D$ with respect to $H.$ Suppose that $|V(D)|\geq 2k+3$ and $|\xi_{D}^{+}(v)|\leq k$ for every $v$ in $V(D)$, for some positive integer $k.$ Then, $S_{H}(D)$ has an $H$-kernel by walks. 
  \end{teo}
Theorem \ref{Hsubdivision} also holds for finite digraphs, so we have the analogous to  Remark \ref{topp} as follows.
\begin{remark}\label{HRRsub}
For every digraph $D$, $0 \leq \kappa^{H}(D) \leq |A(D)|$.
\end{remark}

In this section we will prove that, under certain hypothesis on a digraph $D$, we just need to subdivide the arcs of an  spanning subdigraph of $D$ in order to obtain a kernel on the partial subdivision of $D$. This generalizes Theorem \ref{Hsubdivision} and improves the bound obtained in Remark \ref{HRRsub}.

\subsection{$H$-kernels in the partial subdivision digraph}

Let us define \textit{partial subdivision} of $D$ with respect to $H.$ 

\begin{defi}\label{defi5}
Let $H$ and $D$ be two digraphs, $G$ be a spanning subdigraph of $D$ such that $\delta_{G}^{+}(v)=0$ if and only if $\delta_{D}^{+}(v)=0$ and $S_{H}(G)$ be a subdivision of $G$. A partial subdivision of $D,$ denoted by $R_{\text{\tiny H}}^{\text{\tiny G}}(D)$ is an $H$-colored digraph defined as follows:
\begin{enumerate}
\item $V(R_{\text{\tiny H}}^{\text{\tiny G}}(D))=V(S_{H}(G)).$
\item $A(R_{\text{\tiny H}}^{\text{\tiny G}}(D))=A(S_{H}(G))\cup [A(D)\setminus A(G)].$
\end{enumerate}
Notice that $R_{\text{\tiny H}}^{\text{\tiny G}}(D)$ is an $H$-colored digraph and if $a\in [A(D)\setminus A(G)],$ $c(a)\in V(H).$    
\end{defi}

We need to introduce some notation in order to present our proofs more compactly.
\\
Let $H$ be a digraph and $D$ an $H$-colored digraph. Consider $\{u,v\},S_{1}$ and $S_{2}$ three subsets of $V(D).$ We will write: $u\xrightarrow [\ \ D \ \ ]{H} v$ if there exists a $uv$-$H$-walk in $D$; $S_{1}\xrightarrow [\ \ D \ \ ]{H} S_{2}$ if there is a $S_{1}S_{2}$-$H$-walk in $D$; $u\centernot{\xrightarrow[\ \ D \ \ ]{H}} v$ is the denial of $u\xrightarrow [\ \ D \ \ ]{H} v$; $S_{1}\centernot{\xrightarrow[\ \ D \ \ ]{H}} S_{2}$ is the denial of $S_{1}\xrightarrow [\ \ D \ \ ]{H} S_{2}.$ 
\\

We will define the following digraph that will be useful to prove the main result.
\begin{defi}
Let $H$ and $D$ be two digraphs, $G$ a spanning subdigraph of $D$, $R_{\text{\tiny H}}^{\text{\tiny G}}(D)$ a partial subdivision of $D$. The digraph $D_R$ is such that.
\begin{enumerate}
\item $V(D_{R})= A(G).$
\item $A(D_{R})=\{(a_{i},a_{j}): a_{i}\xrightarrow[R_{\text{\tiny H}}^{\text{\tiny G}}(D)]{H} a_{j}\}. $
\end{enumerate}
\end{defi}

The following result will be useful in this paper.

\begin{teo} \label{1}
\cite{RV} Let D be a digraph possibly infinite. Suppose that $D$ is a transitive digraph such that every infinite outward path has at least one symmetric arc. Then $D$ has a kernel. 
\end{teo}

From now on, we are going to denote by $H$ a digraph possibly with loops, by $D$ a digraph possibly infinite with out loops and by $G$ a spanning subdigraph of $D,$ we will use these digraphs for the construction of $R_{\text{\tiny H}}^{\text{\tiny G}}(D),$ a partial subdivision of $D,$ and the respective digraph $D_{R}.$ We will omit them from the hypothesis of the statements.     

The following notation will simplify some proofs. Let $W=(v_{0},v_{1},\ldots,v_{n})$ be an $H$-walk and $\{v_{i},v_{j}\} \subseteq V(W)$, with $i\leqslant j$, we will denote by $(v_i,W,v_{j})$ any $H$-walk that starts in $v_{i},$ ends in $v_{j}$ and it is contained in $W.$
 
The following lemmas will be very useful for proving the main results.    
\begin{lema} \label{2}

Let $R_{\text{\tiny H}}^{\text{\tiny G}}(D)$ a partial subdivision of $D$ and $\{a,b,c\}\subseteq V(R_{\text{\tiny H}}^{\text{\tiny G}}(D))$ such that $b\in A(D)$. If $W_1$ is an $ab$-$H$-walk in $R_{\text{\tiny H}}^{\text{\tiny G}}(D)$ and $W_2$ is a $bc$-$H$-walk in $R_{\text{\tiny H}}^{\text{\tiny G}}(D)$, then $W_1\cup W_2$ is an $H$-walk in $R_{\text{\tiny H}}^{\text{\tiny G}}(D).$ 
\end{lema}

\begin{proof}
Suppose that $W_{1}=(a=v_{0},v_{1},\ldots,v_{n}=b)$ and $W_{2}=(b=u_{0},u_{1},\ldots,u_{k}=c)$ with $\{k,n\}\subseteq \mathbb{N}$. There are two cases:
\\

\textbf{Case 1.} If $a=b$ or $b=c$, then $W_{1}\cup W_{2}=W_{2}$ or $W_{1}\cup W_{2}=W_{1},$ so $W_{1}\cup W_{2}$   is an $H$-walk in $R_{\text{\tiny H}}^{\text{\tiny G}}(D)$.
\\

\textbf{Case 2.} If $b\notin \{a,c\}$. We have that $W_{1}$ and $W_2$ both are $H$-walks in $R_{\text{\tiny H}}^{\text{\tiny G}}(D)$, now we will show that $(c_{R_{\text{\tiny H}}^{\text{\tiny G}}(D)}(v_{n-1},v_{n}=b),c_{R_{\text{\tiny H}}^{\text{\tiny G}}(D)}(b=u_{0},u_{1}))\in A(H).$ By hypothesis $b\in A(D)$, the definition of $R_{\text{\tiny H}}^{\text{\tiny G}}(D)$ implies that $b=(v_{n-1},u_{1})$, $\delta_{R_{\text{\tiny H}}^{\text{\tiny G}}(D)}^{-}(b)=1=\delta_{R_{\text{\tiny H}}^{\text{\tiny G}}(D)}^{+}(b),$ and $(v_{n-1},v_{n}=b=u_{0},u_{1})$ is an $H$-walk in $R_{\text{\tiny H}}^{\text{\tiny G}}(D)$. Thus $(c_{R_{\text{\tiny H}}^{\text{\tiny G}}(D)}(v_{n-1},v_{n}=b),c_{R_{\text{\tiny H}}^{\text{\tiny G}}(D)}(b=u_{0},u_{1}))\in A(H)$.  
\end{proof}

\begin{lema} \label{3}
If $D$ has no infinite outward path, then $R_{\text{\tiny H}}^{\text{\tiny G}}(D)$ has no infinite outward path. 
\end{lema}
\begin{proof}
Suppose that $T=(v_i)_{i\in \mathbb{N}\cup \{0\}}$ is an infinite outward path in $R_{\text{\tiny H}}^{\text{\tiny G}}(D)$, let $T^{'}=(x_{n_{i}})_{i\in \mathbb{N}\cup \{0\}}$ be the subsequence of $T$ such that $x_{n_{i}}\in V(D)$ for all $i\in \mathbb{N}$ and $[V(T)\setminus V(T^{'})]\cap V(D)=\emptyset$  .
\\

Consider the following cases.
\\

\textbf{Case 1.} If $(x_{n_{i}},x_{n_{i+1}})\in A(T),$ by the definition of $R_{\text{\tiny H}}^{\text{\tiny G}}(D),$ we have that $(x_{n_{i}},x_{n_{i+1}})\in A(D).$ 
\\

\textbf{Case 2.} If $(x_{n_{i}},x_{n_{i+1}})\notin A(T),$ there is a vertex $a_i\in V(R_{\text{\tiny H}}^{\text{\tiny G}}(D)$ such that $(x_{n_{i}},a_i)\in A(T)$ and \linebreak $(a_{i},x_{n_{i+1}})\in A(T),$ and by the definition of $R_{\text{\tiny H}}^{\text{\tiny G}}(D)$ the vertex $a_i$ is the arc $(x_{n_{i}},x_{n_{i+1}})\in A(D).$
\\

Thus $T^{'}=(x_{n_{i}})_{i\in \mathbb{N}\cup \{0\}}$ is an infinite outward path in $D$. We conclude that if $D$ has no infinite outward path, then $R_{\text{\tiny H}}^{\text{\tiny G}}(D)$ has no infinite outward paths.
\end{proof} 

\begin{lema}\label{4}
(Flower \cite{HRR})  
Let $D$ an $H$-colored digraph, with $|V(D)|\geq 2k+3$ and $|\xi_{D}^{+}(v)|\leq k$ for every $v$ in $V(D)$, for some positive integer $k$. Suppose that $\{a_1,a_2,...,a_{2k+2}\}$ is a subset of $V(D)$ such that:
\begin{enumerate}
\item $a_i\neq a_j$ for every $\{i,j\}$ of $\{1,2,\ldots,2k+2\}$ with $i\neq j$.
\item There exists an $a_{i}a_{i+1}$-$H$-walk, say $W_i$, in $D$ for each $i$ in $\{1,\ldots,2k+1\}.$
\item $\bigcap\limits_{i=1}^{2k+1} [V(W_{i})\setminus \{a_{i},a_{i+1}\}]\neq \emptyset.$
\end{enumerate}
Then, there exists a subset $\{i,j\}$ of $\{2,\ldots,2k+1\}$, with $i\leq j-1$, such that there is an $a_{j}a_{i}$-$H$-walk in $D$.
\end{lema}

\begin{lema}\label{5}
 Let $D$ be an $H$-colored digraph and $R_{\text{\tiny H}}^{\text{\tiny G}}(D)$ a partial subdivision of $D$ such that for some positive integer $k,$ $|V(R_{\text{\tiny H}}^{\text{\tiny G}}(D))|\geq 2k+3$ and $|\xi_{R_{\text{\tiny H}}^{\text{\tiny G}}(D)}^{+}(v)|\leq k$ for every $v$ in $V(R_{\text{\tiny H}}^{\text{\tiny G}}(D)).$ Suppose that $(a_{n})_{n\in \mathbb{N}}$ is a sequence of different vertices of $D_R$ such that:
\begin{enumerate}
\item There is an $a_{i}a_{i+1}$-$H$-walk, say $W_i$, in $R_{\text{\tiny H}}^{\text{\tiny G}}(D)$ for every $i\in \mathbb{N.}$

\item $a_{i+1} \centernot{\xrightarrow[R_{\text{\tiny H}}^{\text{\tiny G}}(D)]{H}} a_{i}$ for every $i\in \mathbb{N}$. 
\end{enumerate}

Then, there exist a vertex y in $W_1$, a path P in $R_{\text{\tiny H}}^{\text{\tiny G}}(D)$, a vertex $x$ in $W_{\lambda}$ for some $\lambda$ in  $\mathbb{N}\setminus \{1\}$ and a sequence $(b_m)_{m\in \mathbb{N}}$ of different vertices of $D_R$ such that:

\begin{enumerate}
\item[I.] $x\in A(D).$

\item[II.] $y\in V(W_{\lambda}).$

\item[III.]  There is a $ya_{\lambda +1}$-$H$-walk contained in $W_{\lambda},$ say  $\widehat{W},$ where $x$ is the first vertex in $\widehat{W}$ such that $x\in A(D).$  

\item[IV.] $P$ is an $a_{1}x$-path.

\item[V.] $a_{\lambda +1} \centernot{\xrightarrow[R_{\text{\tiny H}}^{\text{\tiny G}}(D)]{H}} x.$

\item[VI.] $x\xrightarrow [R_{\text{\tiny H}}^{\text{\tiny G}}(D)]{H} a_{\lambda +2}.$

\item[VII.] $x\notin W_{\lambda +n}$ for every $n\in \mathbb{N}\setminus \{1\}.$

\item[VIII.] $b_{1}=x$ and $b_{i}\xrightarrow [R_{\text{\tiny H}}^{\text{\tiny G}}(D)]{H} b_{i+1}$ and $b_{i+1}\centernot{\xrightarrow[R_{\text{\tiny H}}^{\text{\tiny G}}(D)]{H}} b_{i}$ for each $i\in \mathbb{N}.$
 \linebreak
 \item[IX.] For 
\begin{center}
$W_{j}^{'}=\left \{\begin{array}{ccc} \widehat{W}\cup W_{\lambda +1}& if & j=1$ $and$ $x\neq a_{\lambda +1}.\\ W_{\lambda +1} & if & j=1$ $and$ $x=a_{\lambda +1}\\ W_{\lambda +j} & if & j\neq 1.\end{array}\right.$
\end{center}
it holds that for each $j\in \mathbb{N},$
\begin{center}
$V(P)\cap V(W_{j}^{'})= \left\{ \begin{array}{cl}  \{x\} &$ $if$ $j=1.\\ \emptyset & in$ $other$ $case.\end{array}\right.$   
\end{center}
\end{enumerate}
\end{lema}     

\begin{proof}

Let $y\in V(W_1)$ be the first vertex that appears in $\bigcup\limits_{h\in \mathbb{N}\setminus \{1\}} W_{h},$ it exists because $a_2\in V(W_{1})\bigcap V(W_2)$, notice that if $y=a_{f}$ for some $f\in \mathbb{N}\setminus \{1\}$, then $y\neq a_g$ for every $g\in \mathbb{N}\setminus \{f\}$.
\\

Consider the set $I=\{j\in \mathbb{N}: y\in V(W_j)\}.$
\\

\textbf{Claim 1.} $I$ has a maximum element. 

Proceeding by contradiction, suppose that $I$ has no  maximum element. Then we can choose a subset $\{r_{1},r_{2},\ldots,r_{2k+1}\}$ of $I$ such that $r_{1}<r_{2}<r_{3}<\cdots<r_{2k+1}$ and $y\neq a_{r_{i}}$ for every $i\in \{1,2,\ldots,2k+1\}$.
\\

\textbf{Claim 1.1.} For each $i\in \{1,2,\ldots,2k+1\}$, there exists an $a_{r_{i}}a_{r_{i+1}+1}$-$H$-walk in $R_{\text{\tiny H}}^{\text{\tiny G}}(D)$ that contains $y.$ 
\\

Let $i$ be an index in $\{1,2,\ldots,2k+1\}$. Then, by Lemma \ref{2}, $\bigcup\limits_{h=r_{i}+1}^{r_{i+1}} W_{h}$ is an \linebreak $a_{r_{i}+1}a_{r_{i+1}+1}$-$H$-walk in $R_{\text{\tiny H}}^{\text{\tiny G}}(D)$ that contains $y$ since $r_{i+1}$ belongs to $I$.
\\

Now, let $d_{1}=a_{r_{1}}$ and $d_t=a_{r_{t-1}+1}$ for every $t\in \{2,\ldots,2k+2\}$, then:  
\begin{enumerate}
\item $d_{i}\neq d_{j}$ for every $\{i,j\}\subseteq \{1,\ldots,2k+2\}$ with $i\neq j.$

\item There exists a $d_{i}d_{i+1}$-$H$-walk, say $C_{i},$ in $R_{\text{\tiny H}}^{\text{\tiny G}}(D)$ for each $i\in \{1,\ldots,2k+1\}.$

\item $y\in \bigcap\limits_{i=1}^{2k+1} [V(C_{i})\setminus \{d_{i},d_{i+1}\}].$   
\end{enumerate}

Then, it follows from Lemma's Flower that there exists a subset $\{\alpha,\beta\}\subseteq \{2,\ldots,2k+1\}$ with $\alpha \leq \beta -1,$ such that there is a $d_{\beta}d_{\alpha}$-$H$-walk in $R_{\text{\tiny H}}^{\text{\tiny G}}(D),$ say $W'$, notice that $W'$ is an $a_{r_{\beta -1}+1}a_{r_{\alpha -1}+1}$-$H$-walk in $R_{\text{\tiny H}}^{\text{\tiny G}}(D).$ Since $\alpha <\beta -1,$ we have that $\alpha -1<\beta -1$. Thus, $C=\bigcup\limits_{h=r_{\alpha -1}+1}^{r_{\beta -1}-1} W_{h}$ is an $a_{r_{\alpha -1}+1}a_{r_{\beta -1}}$-$H$-walk in $R_{\text{\tiny H}}^{\text{\tiny G}}(D),$ so $W^{'}\cup C$ is an $a_{r_{\beta -1}+1}a_{r_{\beta -1}}$-$H$-walk, and it is a contradiction with $a_{i+1} \centernot{\xrightarrow[R_{\text{\tiny H}}^{\text{\tiny G}}(D)]{H}} a_{i}$ for every $i\in \mathbb{N}$.
\\
Therefore, $I$ has maximum element, say $\lambda$.  Let $\widehat{W}$ be a fixed $H$-walk of the form $(y, W_{\lambda}, a_{\lambda +1})$ and let $\widehat{a}$ the first vertex in $(y, W_{\lambda}, a_{\lambda +1})$ such that $\widehat{a}\in A(D).$  Let $x\in V(W_{\lambda})$ and let 
\\

$$x=\left\{ \begin{array}{ccc} y &  \text{if} &  y\in A(D). \\ \widehat{a}&  \text{if}&  y\notin A(D).\end{array}\right.$$
\\

The choice of $x$ implies that $x\in A(D),$ notice that $x\notin V(W_{\lambda +n})$ for every $n\in \mathbb{N}\setminus \{1\}$, otherwise, if $x\in V(W_{\lambda +m})$ for some $m\in \mathbb{N}\setminus \{1\},$ then $(a_{\lambda +2},\bigcup\limits_{h=\lambda +2}^{\lambda +m}W_{h},x)\cup \widehat{W}$ is an $a_{\lambda +2}a_{\lambda +1}$-$H$-walk in $R_{\text{\tiny H}}^{\text{\tiny G}}(D)$, by Lemma \ref{2}, a contradiction to the hypothesis.
\\

$a_{\lambda +2} \centernot{\xrightarrow[R_{\text{\tiny H}}^{\text{\tiny G}}(D)]{H}} x;$ otherwise, if we suppose that there is an $a_{\lambda +2}x$-$H$-walk in $R_{\text{\tiny H}}^{\text{\tiny G}}(D)$, since $(x,W_{\lambda},a_{\lambda +1})$ is an $H$-walk in $R_{\text{\tiny H}}^{\text{\tiny G}}(D),$ it follows from Lemma \ref{2} that $a_{\lambda +2} \xrightarrow [R_{\text{\tiny H}}^{\text{\tiny G}}(D)] {H} a_{\lambda +1},$ a contradiction.
\\

We are going to define $P$ as an $a_{1}x$-path as follows:
\begin{enumerate}
\item If $y\in A(D)$ then $x=y,$ let $P$ be any $a_{1}x$-path contained in $W_{1}.$
\item If $y\notin A(D),$ let $P$ be any $a_{1}x$-path contained in some $a_{1}x$-$H$-walk of the form $(a_{1},W_{1},y)\cup (y,W_{\lambda},x).$ 
\end{enumerate}

Let $b_1=x,$ let $b_i=a_{\lambda +i}$ for each $i\in \mathbb{N}\setminus \{1\},$ notice that $(b_{m})_{m\in \mathbb{N}}$ is a sequence of different vertices of $D_R$ such that:
\\
\begin{enumerate}
\item $b_{i} \xrightarrow [R_{\text{\tiny H}}^{\text{\tiny G}}(D)] {H} b_{i+1}$ for every $i\in \mathbb{N}.$  
\item $b_{i+1}  \centernot{\xrightarrow[R_{\text{\tiny H}}^{\text{\tiny G}}(D)]{H}} b_{i}$ for every $i\in \mathbb{N}.$
\end{enumerate}

And there is a $b_{1}b_{2}$-$H$-walk in $R_{\text{\tiny H}}^{\text{\tiny G}}(D),$ say $W'_{1},$ that it is contained in some $H$-walk of the form $(x,W_{\lambda},a_{\lambda +1})\cup W_{\lambda +1}$ if $x\neq a_{\lambda +1},$ or in $W_{\lambda +1}$ if $x=a_{\lambda +1}.$ And for:
\\
\begin{center}
$W_{j}^{'}=\left \{\begin{array}{ccc} \widehat{W}\cup W_{\lambda +1}& if & j=1$ $and$ $x\neq a_{\lambda +1}.\\ W_{\lambda +1} & if & j=1$ $and$ $x=a_{\lambda +1}.\\ W_{\lambda +j} & if & j\neq 1.\end{array}\right.$
\end{center}
It holds by the choice of $x$ and the fact that $x\notin V(W_{\lambda +n})$ for every $n\in \mathbb{N}\setminus \{1\},$ that
\\
\begin{center}
$V(P)\cap V(W_{j}^{'})= \left\{ \begin{array}{cl}  \{x\} &$ $if$ $j=1.\\ \emptyset & in$ $other$ $case. \end{array}\right.$   
\end{center}
\end{proof}

\begin{obs}\label{6}
Notice that in Lemma \ref{5} the sequence $(b_{m})_{m\in \mathbb{N}}$ behaves exactly as the sequence $(a_{n})_{n\in \mathbb{N}}.$  
\end{obs}

\begin{lema}\label{7}
Let $D$ be an $H$-colored digraph and $R_{\text{\tiny H}}^{\text{\tiny G}}(D)$ a partial subdivision of $D$ such that for some positive integer $k,$ $|V(R_{\text{\tiny H}}^{\text{\tiny G}}(D))|\geq 2k+3$ and $|\xi_{R_{\text{\tiny H}}^{\text{\tiny G}}(D)}^{+}(v)|\leq k$ for every $v$ in $V(R_{\text{\tiny H}}^{\text{\tiny G}}(D))$. Suppose that $(a_{n})_{n\in \mathbb{N}}$ is a sequence of different vertices of $D_R$ such that: 
\begin{enumerate}
\item There is an $a_{i}a_{i+1}$-$H$-walk, say $W_i$, in $R_{\text{\tiny H}}^{\text{\tiny G}}(D)$ for every $i\in \mathbb{N}$
\item $a_{i+1}\centernot{\xrightarrow [R_{\text{\tiny H}}^{\text{\tiny G}}(D)] {H}} a_{i}$ for every $i\in \mathbb{N}$. 
\end{enumerate}
Then, $R_{\text{\tiny H}}^{\text{\tiny G}}(D)$ has an infinite outward path.

\end{lema}

\begin{proof}
We will use Lemma \ref{5} repeatedly in order to find paths $\{P_m\}$ for every $m \in \mathbb{N}$ such that their union $\cup_{m \in \mathbb{N}} P_m$ is the desired infinite outward path.
\\
We start by renaming the sequence $(a_{i}^{1})_{i\in \mathbb{N}}=(a_{i})_{i\in \mathbb{N}}$ and the walk $W_i=W_i^1$ to facilitate notation.
It follows from Lemma \ref{5} that there exist $P_{1}$ a path of length at least 1 in $R_{\text{\tiny H}}^{\text{\tiny G}}(D),$  $\lambda _{1}\in \mathbb{N}\setminus \{1\},$ a vertex $x_{1}\in V(W_{\lambda_{1}}^{1})$ and a sequence $(a_{r}^{2})_{r\in \mathbb{N}}$ of different vertices of $D_{R}$ satisfying the nine properties of Lemma \ref{5}.

Now, we can again apply Lemma \ref{5} but this time to the sequence $(a^2_n)_{n\in \mathbb{N}}$. Thus, by applying 
Lemma \ref{5} repeatedly, and by  Remark \ref{6}, it follows that for every $m\in \mathbb{N}$, we obtain a sequence of different vertices of $D_{R}$  $(a_{i}^{m})_{m\in \mathbb{N}}$ such that:

\begin{enumerate}
\item There is an $a_{i}^{m}a_{i+1}^{m}$-$H$-walk, say $W_i$, in $R_{\text{\tiny H}}^{\text{\tiny G}}(D)$ for every $i\in \mathbb{N},$
\item $a_{i+1}^{m}\centernot{\xrightarrow [R_{\text{\tiny H}}^{\text{\tiny G}}(D)] {H}} a_{i}^{m}$ for every $i\in \mathbb{N},$ 
\end{enumerate}

where if  $W_{i}^{m}$ is an $a_{i}^{m}a_{i+1}^{m}$-$H$-walk in $R_{\text{\tiny H}}^{\text{\tiny G}}(D),$ then  there exists $P_{m}$ a path of length at least 1 in $R_{\text{\tiny H}}^{\text{\tiny G}}(D),$  $\lambda _{m}\in \mathbb{N}\setminus \{1\},$ a vertex $x_{m}\in V(W_{\lambda_{m}}^{m})$ and a sequence $(a_{r}^{m})_{r\in \mathbb{N}}$ of different vertices of $D_{R}$ such that:
\begin{enumerate}
\item $x_{m}\in A(D).$
\item $P_m$ is an $a_{1}^{m}x_{m}$-path in $R_{\text{\tiny H}}^{\text{\tiny G}}(D).$
\item $P_{m}$ is contained in $\bigcup \limits_{i\in\mathbb{N}} W_{i}^{m}.$
\item $V(P_{m})\cap V((x_{m},W_{\lambda_{m}}^{m},a_{\lambda_{m}+1}^{m}))\cup W_{\lambda_{m}+1}^{m})=x_{m}.$
\item $V(P_{m})\cap V(W_{\lambda _{m}+t}^{m})=\emptyset$ with $t\in \mathbb{N}\setminus \{1\}.$
\item $a_{1}^{m+1}=x_{m}$ and $a_{j}^{m+1}=a_{\lambda _{m}+j}^{m}$ for every $j\in \mathbb{N}\setminus\{1\}.$
\item $ W_{1}^{m+1}=(x_{m}, W_{\lambda _{m}}^{m}, a_{\lambda_{m}+1}^{m})\cup W_{\lambda_{m}+1}^{m}$ is an $a_{1}^{m+1}a_{2}^{m+1}$-$H$-walk in $R_{\text{\tiny H}}^{\text{\tiny G}}(D).$
\item $W_{i}^{m+1}=W_{\lambda_{m}+i}^{m}$ is an $a_{i}^{m+1}a_{i+1}^{m+1}$-$H$-walk in $R_{\text{\tiny H}}^{\text{\tiny G}}(D)$ for every $i\in\mathbb{N}\setminus\{1\}.$
\item $a_{i+1}^{m+1}\centernot{\xrightarrow [R_{\text{\tiny H}}^{\text{\tiny G}}(D)] {H}} a_{i}^{m+1}$ for every $i\in \mathbb{N}.$
\end{enumerate} 

\textbf{Claim 1.} $\bigcup\limits_{i\in\mathbb{N}}P_{i}$ is an infinite outward path in $R_{\text{\tiny H}}^{\text{\tiny G}}(D).$
Let $i$ be an index in $\mathbb{N}.$ Since $P_{i}$ is an $a_{1}^{i}x_{i}$-path in $R_{\text{\tiny H}}^{\text{\tiny G}}(D), $ $P_{i+1}$ is an $a_{1}^{i+1}x_{i+1}$-path in $R_{\text{\tiny H}}^{\text{\tiny G}}(D)$ and $a_{1}^{i+1}=x_{i}$ we have that $P_{i}\cup P_{i+1}$ is a walk in $R_{\text{\tiny H}}^{\text{\tiny G}}(D),$ so $\bigcup\limits_{i\in\mathbb{N}}P_{i}$ is a walk in   $R_{\text{\tiny H}}^{\text{\tiny G}}(D).$
\\

By property 3, for every $i\in \mathbb{N},$ $P_{i+1}$ is contained in $\bigcup \limits_{n\in\mathbb{N}} W_{n}^{i+1}.$ By property 5, $V(P_{i+1})\cap V(W_{\lambda _{t}}^{i+1})=\emptyset$ with $t\in \mathbb{N}\setminus \{1\}$ and by property 4,   $V(P_{i+1})\cap V(W_{1}^{i+1})=x_{i}$ and $P_{i}\cup P_{i+1}$ is a walk in $R_{\text{\tiny H}}^{\text{\tiny G}}(D),$ thus we obtain that $P_{i}\cap P_{i+1}=x_{i}.$
\\

Now we will show that for every $\{i,j\}\subseteq \mathbb{N}$ such that $i+1<j,$ $V(P_{i})\cap V(P_{j})=\emptyset.$
By properties 7 and 8, we have that $\bigcup\limits_{n\in \mathbb{N}}W_{n}^{j}$ is contained in  $\bigcup\limits_{n\in \mathbb{N}}W_{n}^{i+2};$ thus by property 3, $P_{j}$ is contained in $\bigcup\limits_{n\in \mathbb{N}}W_{n}^{i+2}.$ 

Since by property 5, for every $t\in \mathbb{N}\setminus \{1\},$ $V(P_{i})\cap V(W_{t}^{i+1})=\emptyset,$  $ W_{1}^{i+2}=(x_{i+1}, W_{\lambda _{1}+1}^{i+1}, a_{\lambda_{i+1}}^{i+1})\cup W_{\lambda_{i+1}+1}^{i+1},$ $W_{t}^{i+2}=W_{\lambda_{i+1}+t}^{i+1}$  for every $\lambda_{i+1}\in\mathbb{N};$ thus,  $V(P_{i})\cap V(P_{j})=\emptyset.$
\\

Therefore, $\bigcup\limits_{i\in\mathbb{N}}P_{i}$ is an infinite outward path in $R_{\text{\tiny H}}^{\text{\tiny G}}(D).$  

\end{proof}

\begin{prop}\label{8}
Let $D$ a digraph without infinite outward paths and $R_{\text{\tiny H}}^{\text{\tiny G}}(D)$ a partial subdivision of $D.$ Suppose that; for some positive integer $k,$ $|V(R_{\text{\tiny H}}^{\text{\tiny G}}(D))|\geq 2k+3$ and $|\xi_{R_{\text{\tiny H}}^{\text{\tiny G}}(D)}^{+}(v)|\leq k$ for every $v$ in $V(R_{\text{\tiny H}}^{\text{\tiny G}}(D)),$ then $D_R$ has a kernel.
\end{prop}
\begin{proof}
Consider the following claims.
\\

\textbf{Claim 1.} $D_R$ has no asymmetric infinite outward path.
\\
It follows from lemmas \ref{3} and \ref{7}.
\\

\textbf{Claim 2.} $D_R$ is a transitive digraph.
\\
Let $\{u,v,w\}$ be a subset of $V(D_R)$ such that $\{(u,v),(v,w)\}\subseteq A(D_{R}).$ It follows from the definition of $D_R$ that $u\xrightarrow [R_{\text{\tiny H}}^{\text{\tiny G}}(D)] {H} v$ and $v\xrightarrow [R_{\text{\tiny H}}^{\text{\tiny G}}(D)] {H} w.$ Then by Lemma \ref{2} $u\xrightarrow [R_{\text{\tiny H}}^{\text{\tiny G}}(D)] {H} w.$ So, $(u,w)\in A(D_R).$ Thus, $D_R$ is a transitive digraph.
\\

Since every infinite outward path in $D_R$ has a symmetric arc and $D_R$ is a transitive digraph, it follows from Theorem \ref{1} that $D_R$ has a kernel.       
\end{proof}

Now we will determinate when the partial subdivision of $D,$ with respect to $H,$ has $H$-kernel by walks. More over, we will give sufficient conditions for $R_{\text{\tiny H}}^{\text{\tiny G}}(D)$ to have a unique $H$-kernel by walks.  

\begin{teo}\label{9}
Let $D$ be a digraph without infinite outward paths and $R_{\text{\tiny H}}^{\text{\tiny G}}(D)$ a partial subdivision of $D.$ Suppose that; for some positive integer $k,$ $|V(R_{\text{\tiny H}}^{\text{\tiny G}}(D))|\geq 2k+3$ and $|\xi_{R_{\text{\tiny H}}^{\text{\tiny G}}(D)}^{+}(v)|\leq k$ for every $v$ in $V(R_{\text{\tiny H}}^{\text{\tiny G}}(D)),$ then $R_{\text{\tiny H}}^{\text{\tiny G}}(D)$ has an $H$-kernel by walks.
\end{teo}

\begin{proof}
It follows from Proposition \ref{8} that $D_R$ has a kernel $N_1$. Consider the following sets:

\begin{enumerate}
\item[$\bullet$] $B=\{w\in V(D): \delta ^{+}_{R_{\text{\tiny H}}^{\text{\tiny G}}(D)}(w)=0\}.$
\item[$\bullet$] $N_{2}=\{a\in N_{1}: a\xrightarrow [R_{\text{\tiny H}}^{\text{\tiny G}}(D)] {H} w$ for some $w\in B\}$.
\end{enumerate}  
Let $N=B\cup [N_{1}\setminus N_{2}].$
\\

\textbf{Claim 1.} $N$ is an independent set by $H$-walks in $R_{\text{\tiny H}}^{\text{\tiny G}}(D).$
\\

From the definition of $B$ we have that it is an independent set by $H$-walks and $B \centernot{\xrightarrow[R_{\text{\tiny H}}^{\text{\tiny G}}(D)]{H}} [N_{1}\setminus N_{2}],$ and by definition of $N_2,$ $[N_{1}\setminus N_{2}] \centernot{\xrightarrow[R_{\text{\tiny H}}^{\text{\tiny G}}(D)]{H}} B$.
\\

Therefore, it remains to prove that  $[N_{1}\setminus N_{2}]$ is an independent set by $H$-walks in $R_{\text{\tiny H}}^{\text{\tiny G}}(D).$ Proceeding by contradiction, suppose that there exists $\{u,v\}\subseteq [N_{1}\setminus N_{2}]$ such that $u\xrightarrow [R_{\text{\tiny H}}^{\text{\tiny G}}(D)] {H} v,$ since $\{u,v\}\subseteq  N_1$ and $N_{1}\subseteq A(D),$ from the definition of $D_R$ we have that $(u,v)\in A(D_{R})$ and it contradicts the fact that $N_1$ is an independent set in $D_R.$ Thus $[N_{1}\setminus N_{2}]$ is an independent set by $H$-walks in $R_{\text{\tiny H}}^{\text{\tiny G}}(D).$    
\\

\textbf{Claim 2.} $N$ is an absorbent set by $H$-walks in $R_{\text{\tiny H}}^{\text{\tiny G}}(D).$
\\ 
Let $u\in V(R_{\text{\tiny H}}^{\text{\tiny G}}(D))\setminus N.$ 

We will show that $u\xrightarrow [R_{\text{\tiny H}}^{\text{\tiny G}}(D)] {H} w$ for some $w\in N$, consider the following cases:
\\
\begin{enumerate}

\item[Case 1.] $u\in N_2$.

From the definition of $N_2$ there exists $w\in B$ such that $u\xrightarrow [R_{\text{\tiny H}}^{\text{\tiny G}}(D)] {H} w.$

\item[Case 2.] $u\in A(D)$ and $u\notin N_2$.

Since $N_1$ is a kernel of $D_R,$ then there exists a vertex $a\in N_1$ such that $(u,a)\in A(D_R),$ so  $u\xrightarrow [R_{\text{\tiny H}}^{\text{\tiny G}}(D)] {H} a.$ Suppose that $a\in N_2,$ then there exist $x\in B$ such that $a\xrightarrow [R_{\text{\tiny H}}^{\text{\tiny G}}(D)] {H} x$ and it follows from Lemma \ref{2} that $u\xrightarrow [R_{\text{\tiny H}}^{\text{\tiny G}}(D)] {H} x.$ 

\item[Case 3.] $u\in V(D).$

Since $u\notin B$ it follows that $\delta _{R_{\text{\tiny H}}^{\text{\tiny G}}(D)} ^{+} (u)\neq 0,$ by the definition of $R_{\text{\tiny H}}^{\text{\tiny G}}(D)$ at least one arc of $D$ that it is subdivided and it has $u$ as initial vertex, then there exists $z\in V(R_{\text{\tiny H}}^{\text{\tiny G}}(D))\cap A(D)$ such that $(u,z)\in R_{\text{\tiny H}}^{\text{\tiny G}}(D)).$ Suppose that $z\notin [N_{1}\setminus N_{2}],$ then it follows from the two previous cases that there is $x\in N$ such that $z\xrightarrow [R_{\text{\tiny H}}^{\text{\tiny G}}(D)] {H} x,$ and by Lemma \ref{2} $u\xrightarrow [R_{\text{\tiny H}}^{\text{\tiny G}}(D)] {H} x.$ 
\\

Thus $N$ is an absorbent set by $H$-walks in $R_{\text{\tiny H}}^{\text{\tiny G}}(D).$
\\

Therefore, it follows from the Claims 1, 2 and 3 that $N$ is an $H$-kernel by walks in $R_{\text{\tiny H}}^{\text{\tiny G}}(D).$     
\end{enumerate}
\end{proof} 

The following theorem shows that if we add the hypothesis that $D$ has no cycles, we can prove the uniqueness of the $H$-kernel by walks in $R_{\text{\tiny H}}^{\text{\tiny G}}(D).$ 

\begin{teo}\label{10}
 Let $D$ be a digraph without cycles and without infinite outward paths, and let $R_{\text{\tiny H}}^{\text{\tiny G}}(D)$ a partial subdivision of $D.$ Suppose that; for some positive integer $k,$ $|V(R_{\text{\tiny H}}^{\text{\tiny G}}(D))|\geq 2k+3$ and $|\xi_{R_{\text{\tiny H}}^{\text{\tiny G}}(D)}^{+}(v)|\leq k$ for every $v$ in $V(R_{\text{\tiny H}}^{\text{\tiny G}}(D))$, then $R_{\text{\tiny H}}^{\text{\tiny G}}(D)$ has a unique $H$-kernel by walks.
\end{teo}

\begin{proof}
The following Remark will be useful. 

 \begin{obs}\label{obsnociclos}
Since $D$ has no cycles, then $R_{\text{\tiny H}}^{\text{\tiny G}}(D)$ has no cycles. Which implies that $R_{\text{\tiny H}}^{\text{\tiny G}}(D)$ has no closed walks.
\end{obs}

From Theorem \ref{9} we have that $R_{\text{\tiny H}}^{\text{\tiny G}}(D)$ has an $H$-kernel by walks. Let $M$ and $N$ two $H$-kernels by walks in $R_{\text{\tiny H}}^{\text{\tiny G}}(D).$ We will show that $M=N.$

Let $u\in N.$ We are going to prove that $u\in M.$ Proceeding by contradiction, suppose that $u\notin M$ since $M$ and $N$ are $H$-kernels by walks in $R_{\text{\tiny H}}^{\text{\tiny G}}(D),$ then there exists $u_{1}\in M\setminus N$ such that $u\xrightarrow [R_{\text{\tiny H}}^{\text{\tiny G}}(D)] {H} u_{1},$ so there exists $u_{2}\in N\setminus M$ such that $u_{1}\xrightarrow [R_{\text{\tiny H}}^{\text{\tiny G}}(D)] {H} u_{2},$ hence there exists $u_{3}\in M\setminus N$ such that $u_{2}\xrightarrow [R_{\text{\tiny H}}^{\text{\tiny G}}(D)] {H} u_{3}$ and so on. For each $i\in \mathbb{N},$ let $P_i$ be a $u_{i}u_{i+1}$-$H$-walk in $R_{\text{\tiny H}}^{\text{\tiny G}}(D)$  where $u_{0}=u$ when $i=1.$ 
\\
Consider the following claim:
\\

\textbf{Claim.} For every $\{i,j\}\subseteq \mathbb{N}$ we have that:

\begin{center}
$V(P_{i})\cap V(P_{j})= \left\{ \begin{array}{ccl} u_{i+1} & \mbox{if} & j=i+1 \\ \emptyset & \mbox{if} & j\neq i+1 \end{array}\right.$
\end{center}

Let $\{i,j\}\subseteq \mathbb{N}.$ If $j=i+1$ and there exists $x\in V(P_{i})\cap V(P_{j})$ such that $x\neq u_{i+1}.$ Then $(x,P_{i},u_{i+1})\cup (u_{i+1},P_{i+1},x)$ is a closed walk in $R_{\text{\tiny H}}^{\text{\tiny G}}(D),$ and it contradicts Remark \ref{obsnociclos}. Thus $V(P_{i})\cap V(P_{j})=\{u_{i+1}\}.$
\\ 

If $j\neq i+1$ with $i+1<j,$ and there exists $x\in V(P_{i})\cap V(P_{j}),$ then $(x,P_{i},u_{i+1})\cup [\bigcup \limits _{n=i+1}^{j-1} P_{n}]\cup (u_{j},P_{j},x)$ is a closed walk in $R_{\text{\tiny H}}^{\text{\tiny G}}(D),$ which contradicts Remark \ref{obsnociclos}. Thus $V(P_{i})\cap V(P_{j})=\emptyset.$     
\\

So we have that $\bigcup \limits _{n\in \mathbb{N}} P_{n}$ is an infinite outward path in $R_{\text{\tiny H}}^{\text{\tiny G}}(D),$ which contradicts Lemma \ref{3}. 
\\

Thus, $u\in M$ and $M=N$. Therefore $R_{\text{\tiny H}}^{\text{\tiny G}}(D)$ has a unique $H$-kernel by walks. 
\end{proof}

\begin{teo}
Let $D$ be a digraph without infinite outward paths and $R_{\text{\tiny H}}^{\text{\tiny G}}(D)$ a partial subdivision of $D.$ Suppose that; for some positive integer $k,$ $|V(R_{\text{\tiny H}}^{\text{\tiny G}}(D))|\geq 2k+3$ and $|\xi_{R_{\text{\tiny H}}^{\text{\tiny G}}(D)}^{+}(v)|\leq k$ for every $v$ in $V(R_{\text{\tiny H}}^{\text{\tiny G}}(D)),$  then $0 \leq \kappa^{H}(D) \leq |A(G)|$.
\end{teo}

Notice that, if in Theorem \ref{9} we take $D$ as a spanning subdigraph of itself, then we get the subdivision of $D$ with respect to $H.$ Thus Theorem \ref{Hsubdivision} can be obtained as a corollary of Theorem \ref{9}.

\section{The $H$-kernel subdivision number for some families of digraphs}\label{S3}
In this section we will compute $\kappa^{H}(D)$ for a few families of digraphs when $H$ is a digraph in which all of its arcs are loops. Throughout this section  we shall assume that $H$ is the digraph such that $A(H)=\{(v,v): v\in V(H)\}.$ This means that the $H$-walks are monochromatic walks. Let $S$ the color set  of  an $m$-colored digraph $D$, we say that $D$ is \textit{hetero-chromatic} if, for every $i,j\in S,$ $i\neq j.$  

\begin{teo}\label{cn}
Every $m$-colored directed cycle $C_n$ which is not hetero-chromatic has an  $H$-kernel.
\end{teo}

\begin{proof}
    Proceed by induction over the number of monochromatic paths of length at least three.
    For the base case, assume that the directed cycle $C_n$ is $(x_1, \dots, x_n)$ such that the arcs  $(x_1,x_2), (x_2,x_3),\dots,(x_{k-1},x_k) $ have the same color and $(x_k,x_{k+1}) (x_n,x_1)$ have a different color. Consider the set of vertices $K=\{x_n, x_{n-2}, \dots, x_j\}$ where $j\in \{k,k-1\}$, to be specific, $j=k-1$ if $n-k$ is odd and $j=k$ if $n-k$ is even.
    Then $K$ is a $H$-kernel for $C_n.$ 
    
    Now let $G$ be obtained from $C_n$ by substituting a maximal monochromatic path of length at least three with an arc $(x,y).$ Then by induction hypothesis, $G$ has a unique $H$-kernel $K.$ If $y\in K$ then $K $ is also a $H$-kernel for $C_n.$ Otherwise $x \in K$ and we have two cases.
    \begin{enumerate}
        \item Assume $x$ is the final vertex of a monochromatic path $\gamma$ of length at least three. Then $(K\setminus \{x\})\cup N^-(x)$ is an $H$-kernel  for $C_n.$
        \item Let $x_1, x_2, \dots, x_k$ be the remaining vertices of $G$ with the property that the arcs $(x_1,x), (x_2,x_1), \dots, (x_m, x_{m-1})$ are such that no two arcs of the same color have a common vertex. Moreover, we can assume that  $x_m$ is the final vertex of a monochromatic path of length at least three for some $m\leq k$. Then $K$ contains the set $\{x,x_2, x_4, \dots, x_j\}$ with $j \in \{m,m+1\}$, to be specific, $j=m$ if $m$ is even and $j=m+1$ if $m$ is odd.
        Consider the set $M=(K\cup L \cup N^-(y))\setminus \{x,x_2, x_4, \dots, x_j\}$ where $N^-(y)$ is taken in $C_n$ and $$L=\left\{ \begin{array}{ccc}
            \{x_1, x_3, \dots, x_{m+1} & if & m \text{ is even}\\
             \{x_1, x_3, \dots, x_{m} & if & m \text{ is odd}
        \end{array} \right. $$
        Then $M$ is an $H$-kernel for $C_n$. 
         \end{enumerate}
   
\end{proof}

\begin{coro}\label{uniciclo}
Let $D$ be a digraph containing one unique cycle. Then $D$ has an  $H$-kernel unless the cycle is directed, hetero-chromatic, of odd length, and every vertex of the cycle has out degree one.
\end{coro}


\begin{defi}
    We say that a digraph $G$ is a theta-type graph if it is a subdivision of a directed cycle with a chord.
\end{defi}

\begin{teo}
    Let $G$ be a theta-type graph. Then $\kappa^H (G) \leq 1.$
\end{teo}
\begin{proof}
Notice first that the digraph $G$ can be obtained from gluing a directed path $\gamma$ between two vertices $x$ and $y$ of a directed cycle $D.$ 
Proceed by induction over the number of colors in $\gamma$. For the base case assume that $\gamma$ is monochromatic. If $D$ is hetero-chromatic, of odd length and $(u,x)$ has a different color than $\gamma$, let 

$e=\left\{ 
\begin{array}{cc}
    (u,x) &\text{ if }d(x,y)\text{ is odd,} \\
     (x,w) & \text{ if }d(x,y)\text{ is even,}
 \end{array} \right. $ for $u,w \in V(D)$. 
 
 Then $K= \{v \in V(D_e): d(v,y) \text{ is even}\}$ is an $H$-kernel for $G_e$. Now assume that $(u,x)$ has the same color as $\gamma.$ If $d(x,y)$ is odd, then $K= \{v \in V(D_e): d(v,y) \text{ is even}\}\setminus \{x\}$ is an $H$-kernel for $G$. If $d(x,y)$ is even, let $e=(s,u)\in A(D)$ and subdivide this edge by adding the vertex $v$. Then $K= \{v \in V(D_e): d(v,y) \text{ is even}\}\setminus \{x\}$ is an $H$-kernel for $G_e.$

Now we can assume by Theorem \ref{cn}, that $D$ has  an $H$-kernel $K$. Consider first the case when $\gamma=(x,y)$. Then $K$ is an $H$-kernel for $G$ unless $z,y \in K$ where $z$ is such that there exists a 
monochromatic path from $z$ to $x$ which has the same color as $(x,y)$ (in particular we can have $z=x$). If this last case holds, assume $(w,z)\in A(D)$ and subdivide this edge by adding the vertex $v$. Then 
$K\setminus \{z\}\cup\{v\}$ is an $H$-kernel for $G_{(w,z)}.$

This means that we can assume that $\gamma$ has length at least three. Notice that $K$ is either absorbent or independent (or both) in $G$. If $K$ is independent but not absorbent in $G$, we have that $ y \notin K$ and if $z \in K$ is the vertex that absorbs $y$,  then the $yz$-path has a different color than $\gamma$. This means that either $K\cup \{v\}$ is an $H$-kernel for $G$ where $(v,y)\in A(\gamma)$ or there exists a monochromatic path from $t$ to $x$ with the same color as $\gamma$ for some $t\in K$ (we allow $t=x$ if $ x \in K$). Assume that $(a,b), (b,t)\in A(D)$. If $a \notin K$, then $K\setminus \{t\} \cup \{v,b\}$ is an $H$-kernel for $G$. Otherwise, subdivide the edge $e=(b,t)$ by adding a new vertex $c$. Then  $K\setminus \{t\} \cup \{v,c\}$ is an $H$-kernel for $G_c.$

If $K$ is absorbent but not independent in $G$ we have that there exists a $wz$-path in $G$ which is monochromatic and contains $\gamma$ with $z,w\in K.$ Let $(u,w)\in A(D)$ and subdivide this edge by adding the vertex $v.$ Then $K\setminus\{w\}\cup \{v\}$ is an $H$-kernel for $G_{(u,w)}$.

\

For the inductive step, assume that $\gamma= \gamma_1 \cup \dots \cup \gamma_k$ where each $\gamma_i$ is a maximal monochromatic path. Let $G'$ be the graph obtained by coloring $\gamma_{1}$ with the color of $\gamma_2$. Then by induction hypothesis, either $G'$ has an $H$-kernel, or there exists an edge $e$ such that $G'_e$ has an $H$-kernel. In any case let $J \in \{G', G'_e\}$ be the digraph with the $H$-kernel $K$. Now consider $K$ as a subset of the vertices of $G$ or $G_e$ respectively. Let $z$ denote the first vertex of $\gamma_{2}$, thus $z$ is absorbed by $K$ or $ z \in K$. 


 If $ z\notin K$, notice that $ x\notin K$ and assume first that $\gamma_{1}$ has a different color than $(w,x)$. If $\gamma_1$ has length at least three, take $v \in N^-(z),$ thus $K\cup \{v\}$ is an $H$-kernel for $G$ or $G_e$ respectively.  If $\gamma_{1}=(x,z)$ and $x$ is absorbed by $K$ we are done.  If $J=G'$ we can subdivide $(x,z)$ by adding a new vertex $u$ and thus $K\cup \{u\}$ is an $H$-kernel for $G_{(x,z)}.$   If $J=G'_e$, $\gamma_1=(x,z)$ but $ x $ is absorbed by some vertex in $K$ we are done.
 
If $\gamma_1$ and $(w,x)$ have the same color but $w \notin K$, then if $x\in K$ in which case $K\setminus \{x\} \cup N^-(z)$ is an $H$-kernel for $G$ or $G'$ respectively. If $x\notin K$, then $w$ and $x$ are absorbed by some vertex in $K$ in which $K\cup N^-(z)$ is an $H$-kernel for $G$ or $G'$ respectively when $ x \neq N^-(z)$ and $ K$ is a kernel for $G$ or $G'$ respectively when $x=N^-(z).$

If $ z \in K$ then either $K$ is also an $H$-kernel for $G$ or $G_e$ respectively or $K$ is absorbent but not independent.
Thus we have three cases remaining:
\begin{itemize}
    \item $z\in K$, $w=N^-(x)\in K$ and the $wz$-path is monochromatic
    \item $z\notin K$, $\gamma_1$ and $(w,x)$ have the same color with $w\in K$
    \item $z\notin K$, $J=G'_e$, $(x,z)=\gamma_1$ and $(w,x)$ have different colors, and $x$ is not being absorbed by any vertex in $K$
    
\end{itemize}

  In the second case, since $x \notin K$ and $x$ is not absorbed by any vertex of $K,$ neither is its in-neighbour thus $w=N^-(x)\in K$ in all three cases. In cases one and two, if $J=G'$, we may subdivide $d=(N^-(w), w)$ by adding the vertex $u$ and $K\setminus w \cup \{u, N^-(z)\}$ (if $ z \notin K$) or $K\cup u\setminus w $ (if $z \in K$) is an $H$-kernel for $G_d$ when $d$ has a different color than $(w,x).$
  
  In all three cases, if $d$ and $(w,x)$ have the same color, $K\setminus w \cup N^-(z)$ is an $H$-kernel if $ z\notin K$ and $K\setminus w$ is an $H$-kernel if $z\in K.$  Thus we may assume that in all three cases $J=G'_e.$
  
  In all three cases we want to remove $w$ from $K$, and add new vertices to $K\setminus w$ to obtain an $H$-kernel. In the cases when $z\notin K$ we are going to add $N^-(z)$ to $K.$ We now have several cases:
\begin{enumerate}
    \item the $yw$-path contains a monochromatic path of length at least three
    \item both $\gamma$ and the $xy$-path contained in $D$ contain a monochromatic path of length at least three
    \item $G_e$ contains a hetero-chromatic cycle which is either $D$ or $\gamma \cup \alpha$ where $\alpha $ is the $yx$-path.
\end{enumerate}
In Cases 1. and 2. we remove vertices from $K$ and substitute them with their in-neighbour (or in-neighbours if it is  the vertex $y$) one by one starting with $w$ until we arrive at a monochromatic path and we obtain an $H$-kernel.

In Case 3., if the directed cycle is of even length we can do the process described in Case 2. but instead of stopping at a monochromatic path of length at least three, we continue until we obtain an $H$-kernel (in one direction we do encounter a monochromatic path of length at least three).
Otherwise we forget about $K$ and we are going to build an $H$-kernel as follows. Let $C\in \{D, \gamma \cup \alpha\}$ be the directed cycle and let $v$ be a vertex which absorbs $x$ and is not in $C$. Since  $C\setminus x$ is a path, we have that  $C\setminus x$ has an $H$-kernel $M$. Since $G_e \setminus C$ does contain a monochromatic path of length at least three, we can extend the set $M\cup \{v\}$ to  an $H$-kernel of $G_e$ unless the $vy$-path is monochromatic in which case $G$ or $G'$ has a kernel.

\end{proof}

\begin{defi}
  Let $J$ be a digraph and $G_1,\dots, G_m$ be digraphs. Define a split-root product of $J$ with the family $\{G_i\}_{i=1}^m$ as follows. First select $m$ vertices of $J$, $v_1, \dots, v_m$ and let $J'$ be the digraph obtained from $J$ by splitting every distinguished vertex $v_i$ into $d(v_i)$ vertices of degree one. For each $G_i$ select $d(v_i)$ vertices and identify each of them with one of the $d(v_i)$ vertices of degree one in $J'$ which correspond to $v_i$ in $J$ for $1 \leq i \leq m.$ The obtained graph is a split-root product of $J$ with $\{G_i\}_{i=1}^m.$ See for example Figure \ref{splitroot}.
\end{defi}

\begin{figure}[h!]
	\centering
	\begin{tikzpicture}[every node/.style={circle, draw, scale=.6}, scale=1.0, rotate = 180, xscale = -1]

		\node (1) at ( 10.81, 3.47) {};
		\node (2) at ( 11.32, 4.21) {};
		\node (3) at ( 10.7, 4.95) {};
		\node (4) at ( 9.64, 4.96) {};
		\node (5) at ( 9.03, 4.11) {};
		\node (6) at ( 9.57, 3.38) {};
		\node (7) at ( 10.08, 4.19) {};
		\node (8) at ( 7.11, 2.79) {};
		\node (9) at ( 6.56, 3.57) {};
		\node (10) at ( 5.54, 3.52) {};
		\node (11) at ( 4.95, 2.65) {};
		\node (12) at ( 5.59, 1.86) {};
		\node (13) at ( 6.59, 1.94) {};
		\node (14) at ( 7.83, 3.48) {};
		\node (15) at ( 1.84, 4.06) {};
		\node (16) at ( 2.16, 3.17) {};
		\node (17) at ( 3.26, 3.29) {};
		\node (18) at ( 3.6, 4.27) {};
		\node (19) at ( 3.15, 5.14) {};
		\node (20) at ( 2.11, 5.27) {};
		\node (21) at ( 2.51, 4.64) {};
		\node (22) at ( 9.41, 0.62) {};
		\node (23) at ( 8.87, 1.22) {};
		\node (24) at ( 9.88, 1.23) {};
		\node (25) at ( 7.25, 1.23) {};
		\node (26) at ( 7.85, 0.73) {};
		\node (27) at ( 11.2, 2.63) {};
		\node (28) at ( 1.89, 2.21) {};
		\node (29) at ( 4.65, 5.32) {};

		\draw[->] (22) -- (23);
		\draw[->] (24) -- (22);
		\draw[->] (23) -- (24);
		\draw[->] (1) -- (6);
		\draw[->] (2) -- (1);
		\draw[->] (3) -- (2);
		\draw[->] (4) -- (3);
		\draw[->] (5) -- (4);
		\draw[->] (6) -- (5);
		\draw[->] (8) -- (13);
		\draw[->] (9) -- (8);
		\draw[->] (10) -- (9);
		\draw[->] (11) -- (10);
		\draw[->] (12) -- (11);
		\draw[->] (13) -- (12);
		\draw[->] (18) -- (17);
		\draw[->] (19) -- (18);
		\draw[->] (20) -- (19);
		\draw[->] (15) -- (20);
		\draw[->] (16) -- (15);
		\draw[->] (17) -- (16);
		\draw[->] (7) -- (6);
		\draw[->] (3) -- (7);
		\draw[->] (21) -- (15);
		\draw[->] (19) -- (21);
		\draw[->] (14) -- (8);
		\draw[->] (5) -- (14);
		\draw[->] (25) -- (13);
		\draw[->] (25) -- (26);
		\draw[->] (26) -- (23);
		\draw[->] (17) -- (11);
		\draw[->] (29) -- (19);
		\draw[->] (16) -- (28);
		\draw[->] (1) -- (27);

	\end{tikzpicture}
	\caption{A split-root product of a tree with two directed cycles and two theta-type graphs}
	\label{splitroot}
\end{figure}
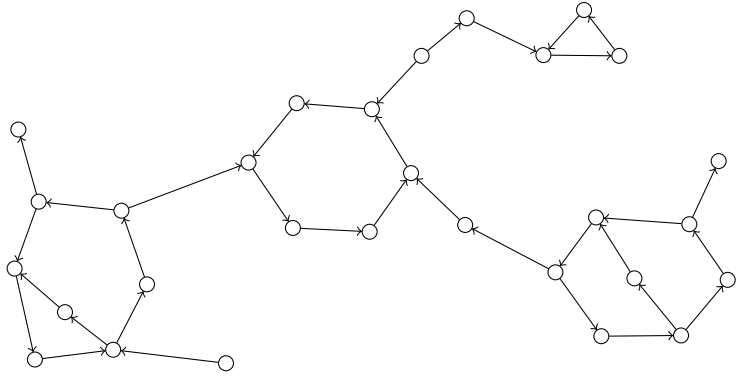

Denote by $N_c^-(v)$ and by $N_c^+(v)$ the color open neighbourhoods of a vertex $v$, this is $$N_c^-(v)=\{w \in V(G): \text{there exists a monochromatic path from }w \text{ to }v\}$$ and $$N_c^+(v)=\{w \in V(G): \text{there exists a monochromatic path from }v \text{ to }w\}$$ and analogously we can define the closed color neighborhoods as $N_c^+[v]=N_c^+(v)\cup \{v\}$ and $N_c^-[v]=N_c^-(v)\cup \{v\}.$ Given a subset $A \subset V(G)$ we denote $N_c^-[A]=\cup_{a\in A}N_c^-[A]$ and $N_c^+[A]=\cup_{a\in A}N_c^+[A]$.

\begin{teo}
    Let $G$ be a split-root product of a tree $T$ with $\{G_i\}_{i=1}^n.$
    Then $\kappa^H(G) \leq \sum\limits_{i=1}^{n} \kappa^H(G_i).$
\end{teo}
\begin{proof}
    Let $\Lambda_i$ denote a set of arcs such that $G_{i_{\Lambda_i}}$ has an  $H$-kernel $K_i$, and let $v_1,\dots, v_n$ denote the distinguished vertices of $T$.
    Proceed by induction over $n.$ For the base case, consider $G_{\Lambda_1}$ and notice that $K_1$ can be extended to an  $H$-kernel of $G$, thus $\kappa^H(G)\leq \kappa^H(G_1).$
    For the inductive step, we have two cases.
    
    If $d^-(v_i)=0$ in $T$ for some $1 \leq i \leq n$, assume that $i=1$ and consider $G_{\Lambda_1}.$ Let $G'$ be the digraph obtained from $G_{\Lambda_1}$ by removing $N_c^-[K_1]$. Then $G'$ is a possibly disconnected digraph such that each component is obtained from a tree by substituting less than $n$ vertices with graphs $G_i$. Thus by induction hypothesis, $\kappa^H(G')\leq \sum\limits_{i=2}^n\kappa^H(G_i)$ and assume that $G'_{\Lambda'}$ has an $H$-kernel $K'$ where $\Lambda'= \bigcup\limits_{i=2}^n\Lambda_i.$ Then $K'\cup K_1$ is an  $H$-kernel for $G_{\Lambda}$ where $\Lambda=\Lambda'\cup \Lambda_1$ thus $|\Lambda|\leq \sum_{i=1}^n \kappa^H(G_i).$

    If $d^-(v_i)>0$ for $1\leq i \leq n$, let $J$ denote the vertices of $T$ which have out-degree zero. Then $G\setminus N_c^-[J]$ is a possibly disconnected digraph such that each component is obtained from a tree by substituting less than $n$ vertices with graphs $G_i$ (otherwise we take again the vertices of out-degree zero in  $G\setminus N_c^-[J]$ and remove them together with their color in-neighborhood). Thus by induction hypothesis, $(G\setminus N_c^-[J])_{\Lambda}$ has an  $H$-kernel $K$ with $|\Lambda|\leq \sum\limits_{i=1}^n |\Lambda_i|.$ Then $K\cup J$ is an  $H$-kernel for $G_{\Lambda}$ thus $\kappa^H(G) \leq \sum\limits_{i=1}^n \kappa^H(G_i).$
\end{proof}

\begin{coro}
Let $D$ be a split-root product of a tree $T$ with a family of $n$ digraphs which are directed cycles or theta-type graphs. Then $\kappa^H(G)\leq n.$
\end{coro}

\begin{ejem}
    Given any coloring on the edges of the digraph $G$ shown in Figure \ref{splitroot}, we have $\kappa^H(G)\leq 4.$
\end{ejem}

\end{document}